\documentclass{article}
\usepackage[utf8]{inputenc}
\usepackage{amsmath}
\usepackage{amsfonts}
\usepackage{amsthm}
\usepackage{amssymb}
\usepackage[unicode=true,pdfusetitle,
 bookmarks=true,bookmarksnumbered=false,bookmarksopen=false,
 breaklinks=false,pdfborder={0 0 1},backref=page,colorlinks=true]{hyperref}
\usepackage{graphicx}
\usepackage{adjustbox}
\usepackage{color}
\usepackage{wrapfig}
\usepackage{float}
\usepackage{geometry}
\usepackage{caption}
\usepackage{subcaption}
\usepackage[all,cmtip]{xy}

\newtheorem{thm}{Theorem}[section]
\newtheorem{prop}[thm]{Proposition}
\newtheorem{cor}[thm]{Corollary}
\newtheorem{con}[thm]{Conjecture}
\newtheorem{lemma}[thm]{Lemma}
\newtheorem{claim}[thm]{Claim}
\theoremstyle{definition}
\newtheorem{define}[thm]{Definition}
\theoremstyle{remark}
\newtheorem*{remark}{Remark}
\newtheorem*{example}{Example}
\numberwithin{equation}{section}
\numberwithin{figure}{section}

\DeclareMathOperator{\Aut}{Aut}

\DeclareMathOperator{\Crit}{Crit}
\DeclareMathOperator{\CommCrit}{CommCrit}
\DeclareMathOperator{\crit}{crit}
\DeclareMathOperator{\MCG}{MCG}
\DeclareMathOperator{\cl}{cl}
\DeclareMathOperator{\scl}{scl}

\DeclareMathOperator{\tr}{tr}

\DeclareMathOperator{\rk}{rk}

\newcommand{\QQ}{\mathbb{Q}}
\newcommand{\EE}{\mathbb{E}}

\newcommand{\CC}{\mathbb{C}}
\newcommand{\ZZ}{\mathbb{Z}}
\newcommand{\imto}{\looparrowright}

\newcommand{\s}{\sigma}
\renewcommand{\S}{\Sigma}
\newcommand{\A}{\mathcal{A}}

\newcommand{\onto}{\twoheadrightarrow}
\newcommand{\into}{\hookrightarrow}
\newcommand{\grp}[1]{\langle #1 \rangle}
\newcommand{\Ga}{\Gamma}
\newcommand{\Om}{\Omega}
\graphicspath{ {./images/} }

\title{Word measures on unitary groups: improved bounds for small representations}
\author{Yaron Brodsky}
\date{\today}

\begin{document}
\maketitle

\begin{abstract}
Let $F$ be a free group of rank $r$ and fix some $w\in F$. For any compact group $G$ we can define a measure $\mu_{w,G}$ on $G$ by (Haar-)uniformly sampling $g_1,...,g_r\in G$ and evaluating $w(g_1,...,g_r)$. In \cite{PuderMageeUn}, Magee and Puder studied the case where $G$ is the unitary group $U(n)$, and analyzed how the moments of $\mu_{w,U(n)}$ behave as a function of $n$. In particular, they obtained asymptotic bounds on those moments, related to the commutator length and the stable commutator length of $w$. We continue their line of work and give a more precise analysis of the asymptotic behavior of the moments of $\mu_{w, U(n)}$, showing that it is related to another algebraic invariant of $w$: its primitivity rank.

In addition, we prove a special case of a conjecture of Hanany and Puder (\cite[Conjecture 1.13]{Hanany}) regarding the asymptotic behaviour of expected values of irreducible characters under $\mu_{w, U(n)}$.
\end{abstract}

{
  \setcounter{tocdepth}{2}
  \hypersetup{linkcolor=black}
  \tableofcontents
}

\section{Introduction}
Let $F_r$ be the free group on $r$ generators. Given a word $w\in F_r$, we can define, for any compact group $G$, a measure $\mu_{w, G}$, by choosing $r$ elements of $G$ Haar-uniformly and independently and plugging into $w$. In other words, $w$ induces a map $G^r \to G$ by $(g_1,...,g_r)\mapsto w(g_1,...,g_r)$, and we consider the push-forward of the Haar measure on $G^r$ by this map. It can be shown that $\mu_{w,G}=\mu_{\phi(w),G}$ for any $\phi\in \Aut(F_r)$ (see \cite[Fact 2.5]{MageeePuder2015} for a proof). Namely, $\mu_{w,G}$ only depends on the orbit of $w$ under the action of $\Aut(F_r)$. The converse statement is a well-known conjecture in the field of word measures and word maps (see \cite[Conjecture 1.3]{MPSurfaceWords}) :

\begin{con}
If $\mu_{w, G} = \mu_{w', G}$ for every compact group $G$, then $w$ and $w'$ are in the same $\Aut(F_r)$-orbit.
\end{con}

\begin{remark}
This conjecture usually appears in the literature in a stronger form, where one only assumes $w$ and $w'$ induce the same measure on all \emph{finite} groups. See \cite[Question 2.2]{AV}, \cite[Conjecture 4.2]{ShalevWordMaps} and \cite[Section 8]{PuderParWordMeasures}.
\end{remark}

One approach to this problem, which is also a question of independent interest, is to study the Fourier coefficients, or moments, of $\mu_{w,G}$. By this we mean that we study $\int_G f d\mu_{w, G}$ for various functions $f$. Since $\mu_{w,G}$ is easily seen to be invariant under conjugation, it is in fact enough to consider class functions on $G$ (that is, functions whose value at an element depends only at its conjugacy class). It turns out that in many cases, the coefficients, or moments, encode interesting algebraic properties of $w$. For example, the case of $G = S_n$ was studied in \cite{PuderParWordMeasures} and in \cite{Hanany}, where it was shown that the behavior of the moments of $w$ on $S_n$ for large $n$ depends on a certain notion called the \emph{primitivty rank} of $w$ (first introduced in \cite{PuderPiw}, see Definition \ref{def:piw}). The case of $G = U(n)$ was studied by Magee and Puder in \cite{PuderMageeUn}, where the behaviour of the moments were shown to be related to the classical notion of the commutator length of $w$ (see Definition \ref{def:cl}).

In this article, we continue the line of work started in \cite{PuderMageeUn} and focus on the case of unitary groups, $G = U(n)$. Magee and Puder showed that the moments of $\mu_{w,U(n)}$ are rational functions of $n$ and a geometric description of their Laurent series expansions (in $\frac1n$) was given. We use this description to analyze the asymptotic behavior and dominant terms of those moments. Before stating our results, we recall the definitions of primitivity rank and commutator length, and introduce the notion of commutator-critical subgroups:
\begin{define}\label{def:piw}
The \emph{primitivity rank} of a word $w\in F_r$, denoted $\pi(w)$, is the minimal rank of a subgroup $H \le F_r$ containing $w$ as a non-primitive element. If there are no such subgroups we set $\pi(w) = \infty$.
\end{define}

\begin{example}
For a primitive element, $\pi(w)=\infty$ (see \cite[Lemma 4.1]{PuderPiw}). We have that $\pi(w) = 0$ if and only if $w = 1$. The next simplest case is $\pi(w)=1$ --- this holds precisely when $w$ is a proper power. Indeed, $w=u^d$ with $|d|\ge 2$ if and only if it $w\neq 1$ and $w$ is imprimitive in $\grp{u}$. Finally, all non-primitive elements of $F_r$ clearly satisfy $\pi(w) \le r$.
\end{example}

\begin{define}\label{def:cl}
The \emph{commutator length} of a word $w\in F_r$, denoted $\cl(w)$, is the least possible $g$ for which we can write $w$ as a product of $g$ commutators. If $w\notin [F_r, F_r]$ we set $\cl(w)=\infty$.
\end{define}

\begin{define}\label{def:commcrit}
A \emph{commutator-critical subgroup} of $w$ is a subgroup $H \le F_r$ of rank $\pi(w)$ containing $w$ as a standard surface word --- that is, there is a basis $u_1,v_1,...,u_{\pi(w)/2},v_{\pi(w)/2}$ of $H$ such that
$$
w = [u_1, v_1] \cdots [u_{\pi(w)/2}, v_{\pi(w)/2}].
$$
(In particular, there are no commutator-critical subgroups if $\pi(w)$ is odd.) We denote the set of all commutator-critical subgroups of $w$ by $\CommCrit(w)$.
\end{define}

\begin{remark}
Compare this with the notion of \emph{critical} subgroups, introduced in \cite{PuderPiw}: a critical subgroup of $w$ is a subgroup of rank $\pi(w)$ containing $w$ as a non-primitive element.
\end{remark}

\begin{remark}
The cardinality of $\CommCrit(w)$ can be described as the number of ways to write $w$ as a product of $\pi(w)/2$ commutators of $\pi(w)$ free words, up to a certain equivalence relation - see Subsection \ref{subsec:alg-intro}.
\end{remark}

Note that if $w=[u_1, v_1]\cdots [u_g, v_g]$, then $w$ is a non-primitive element of the subgroup generated by the $u_i$ and $v_i$. Since $\rk \grp{u_1,\dots,u_g, v_1\dots, v_g}\le 2g$, we have:
\begin{claim}\label{prop:picl}
For any $w\in F_r$, $\pi(w) \le 2\cl(w)$. In particular, $w$ has no commutator-critical subgroups unless $\pi(w) = 2\cl(w)$.
\end{claim}

We now introduce some notation. Denote (by an abuse of notation) $\mu_w = \mu_{w,U(n)}$, where we think of $n$ as a parameter. For any family of functions $f=(f_n)_{n=1}^{\infty}$ (where $f_n
\in L^2(U(n))$) we write $\EE_w[f] = \int_{U(n)} f d\mu_w$ (note that this is actually a function of $n$). For any nonzero $m\in\ZZ$, let $\xi_m$ be the family of class functions on $U(n)$ defined by
$$
\xi_m(A):= \tr(A^m).
$$
(Note that since $A^{-1} = A^*$, we have that $\xi_{-m}(A) = \overline{\xi_m(A)}$.)

The mixed moments of $\mu_w$ are then
$$
\EE_w[\xi_{m_1}\cdots\xi_{m_\ell}] = \int_{U(n)} \tr(A^{m_1}) \cdots \tr(A^{m_\ell}) d\mu_w(A).
$$

\begin{remark}
In the notation of \cite{PuderMageeUn}, this is $\mathcal{T}r_{w^{m_1},\cdots, w^{m_\ell}}$.
\end{remark}

Finally, observe that there is a natural inner product $\grp{\cdot,\cdot}_n$ on $L^2(U(n))$. Consider $f=\xi_{m_1}\cdots \xi_{m_\ell}$ and $g=\xi_{n_1}\cdots \xi_{n_k}$ for some $m_1,...,m_\ell, n_1,...,n_k\in\ZZ$. According to \cite[Theorem 2]{DiaconisShershahani}, the inner product $\grp{f,g}_n$ (when $f$ and $g$ are considered as functions on $U(n)$) is independent of $n$ for large enough $n$. Hence we denote by $\grp{f, g}$ the limit $\lim_{n\to\infty} \grp{f,g}_n$ and think of it as a constant.

Magee and Puder showed that each $\EE_w[\xi_{m_1}\cdots\xi_{m_\ell}]$ is a rational functions of $n$, for large enough $n$. Furthermore, it is shown that it admits the following remarkable power series expansion, for large enough $n$:
$$
\EE_w[\xi_{m_1}\cdots \xi_{m_\ell}] = \sum_{(\S, f)\; \text{admissible}} \chi^{(2)}(\MCG(f))\cdot n^{\chi(\S)},
$$
where the sum is taken over all equivalence classes of \emph{admissible pairs} $(\S, f)$, which are pairs of a compact oriented surface $\S$ with $\ell$ boundary components $\partial_1, ..., \partial_\ell$, and a map $f:\S\to\bigvee^r S^1$ mapping $\partial_i$ to $w^{m_i}$ (w.r.t. some fixed base points on the boundaries). The terms $\chi^{(2)}(\MCG(f))$ are certain coefficients depending on $(\S, f)$. We review this in more depth in Section \ref{sec:recap}.

In particular, the behavior of $\EE_w[\xi_1] = \EE_w[\tr]$ is bounded by the largest possible Euler characteristic of an admissible surface representing $w$. It is a classical result that the genus of such a surface is exactly $\cl(w)$ (see for example \cite[Proposition 1.1]{Culler}). Hence,
$$
\EE_w[\tr] = O(n^{1-2\cl(w)}).
$$

This is \cite[Corollary 1.8]{PuderMageeUn}. They also describe the coefficient of $n^{1-2\cl(w)}$: it is a certain weighted count of the number of solutions to $w=[u_1, v_1]\cdots [u_g, v_g]$ for $g=\cl(w)$.

\subsection{Main results}

We use the power series expansion of \cite{PuderMageeUn} to give a more accurate analysis of the asymptotic behaviour of the moments of $\mu_w$:

\begin{thm}\label{thm:EwTpi}
Let $w \in F_r$ be a non-power, let $m_1,...,m_\ell \in \ZZ$ be nonzero integers, and denote $T=\xi_{m_1}\cdots \xi_{m_\ell}$. Then for large enough $n$,
$$
\EE_w[T] = \grp{T, 1} + \left(\grp{T, \tr}+\grp{T,\overline{tr}}\right)\cdot |\CommCrit(w)|\cdot n^{1-\pi(w)} + O(n^{-\pi(w)}).
$$
\end{thm}

In particular, by Claim \ref{prop:picl}, if $\pi(w) < 2\cl(w)$ we have:
$$
\EE_w[T] = \grp{T, 1} + O(n^{-\pi(w)})
$$

\begin{remark}
In \cite[Theorem 1.3]{Hanany}, a similar result is proven, for the case of $S_n$ instead of $U(n)$. However, the method of proof is different. See Subsection \ref{subsec:othergrps} for further details.
\end{remark}
\begin{remark} Observe that, for large enough $n$, $\grp{T, 1} = \grp{T,1}_n = \EE_{unif}[T]$, the Haar-uniform expectation of $T$. This is also equivalent to $\EE_x[T]$, the expectation of $T$ under the word measure of a primitive word $x\in F_r$. Our proof will use this latter interpretation.
\end{remark}

A key role in our proof is played by a dependence theorem by Louder and Wilton, proved in \cite{WiltonLouder}. We also prove a slightly weaker result using a simpler, more direct method. We decided to include the latter in the article since the proof is based on a different idea, which we believe might be interesting on its own:

\begin{thm}\label{thm:EwT}
Let $w \in F_r$ and let $m_1,...,m_\ell \in \ZZ$ be nonzero integers, and let $T$ denote $\xi_{m_1}\cdots \xi_{m_\ell}$. Then
$$
\EE_w[T] = \grp{T, 1} + O(n^{1-\pi(w)}).
$$
\end{thm}
\begin{remark}
This holds trivially if $w$ is a proper power, since $1-\pi(w)=0$ in this case.
\end{remark}

We also give an improved bound for the case $T=\xi_1 \xi_{-1}=\tr(A)\tr(A^{-1})$. It can be shown that $\grp{\xi_1 \xi_{-1}, 1} = 1$ (see Section \ref{sec:pfthm2}), and we prove:
\begin{thm}\label{thm:EwX2}
Let $w\in F_r$. For $T = \xi_1\xi_{-1} = \tr(A)\tr(A^{-1})$, we have
$$
\EE_w[T] = 1 + O(n^{2(1 - \pi(w)})).
$$
\end{thm}
This is interesting as it is a special case of Conjecture \ref{con:dimchi}, as we explain below.

\subsection{Stable class functions}
One can try and study more general class functions, and not just moments. Clearly, we cannot obtain bounds for general class functions, since we need to incorporate the dependence on $n$. Hence, we develop a notion of a \emph{family} of class functions, or, equivalently, stable class functions. We restrict ourselves to "rational" class functions: the ones that can be represented as rational functions of the eigenvalues.

Let $\A$ be the infinite-dimensional polynomial ring $\CC[\xi_1,\xi_2,\dots,\xi_{-1},\xi_{-2},\dots]$. We think of it as the ring of stable (rational) class functions, as follows. A stable class function should be a symmetric function of the eigenvalues, and indeed $\A$ admits such an interpretation: first, consider the algebra of stable symmetric polynomials. This is defined as $\varprojlim \CC[x_1,...,x_n]^{S_n}$, where $\CC[x_1,...,x_n]^{S_n}$ is the ring of symmetric polynomials (the elements fixed under the action of $S_n$), and the limit is taken over the maps $\CC[x_1,...,x_{n+1}]^{S_{n+1}}\to \CC[x_1,...,x_n]^{S_n}$ given by $x_{n+1}\mapsto 0$. According to \cite[Proposition I.2.12]{macdonald}, we have an isomorphism 
$$
\CC[\xi_1,\xi_2,...] \cong \varprojlim \CC[x_1,...,x_n]^{S_n}.
$$
Now, for stable rational functions, we can write (cf. \cite{Koike})

$$
\A \cong \CC[\xi_1,\xi_2,...] \otimes \CC[\xi_{-1},\xi_{-2},...] \cong \left(\varprojlim \CC[x_1,...,x_n]^{S_n}\right) \otimes \left(\varprojlim \CC[x_1^{-1},\dots, x_n^{-1}]^{S_n}\right).
$$

For any $n$, there is a map $\A \to C(U(n))$, where $C(U(n))$ is the space of class functions, given by $\xi_m \mapsto \tr(A^m)$. In particular, for any $n$ there is an inner product $\grp{,}_n$ on $\A$, obtained from the standard $L^2$ inner product on $C(U(n))$. However, by \cite[Theorem 2]{DiaconisShershahani}, for any two fixed $f, g\in\A$, $\grp{f, g}_n$ stabilizes for large enough $n$. Hence, there is a well-defined stable inner product $\grp{,}=\grp{,}_\infty$ on $\A$.

By the results of \cite{PuderMageeUn}, there is a well-defined function $\EE_w:\A \to \QQ(n)$. The following result is an immediate consequence of Theorem \ref{thm:EwTpi}:

\begin{cor}\label{cor:Ewf}
Let $f \in \A$ be a stable rational class function. Then for any non-power $w\in F_r$,
$$
\EE_w[f] = \grp{f, 1} + \grp{f, \xi_1 + \xi_{-1}} \cdot |\CommCrit(w)| \cdot n^{1 - \pi(w)} + O(n^{-\pi(w)}).
$$
\end{cor}

Arguably, the most important class functions are the irreducible characters (and in fact the irreducible characters of a group $G$ form a basis for the class functions on $G$). This leads us to define the notion of stable families of irreducible characters.

Recall that the irreducible characters of $U(n)$ are in bijection with $n$-tuples of non-increasing integers $(\lambda_1,...,\lambda_n) \in \ZZ^n$ (see \cite{FultonHarris}). It is convenient to denote them by pairs of partitions $(\lambda, \mu)$ with $l(\lambda) + l(\mu) \le n$: if $\lambda = (\lambda_1,\dots, \lambda_s)$ and $\mu = (\mu_1, \dots, \mu_r)$ with $\lambda_i \ge \lambda_{i + 1} > 0$ and $\mu_j \ge \mu_{j + 1} > 0$, we denote by $[\lambda, \mu]_n$ the sequence of length $n$
$$
[\lambda, \mu]_n := (\lambda_1,\dots,\lambda_s, 0, \dots, 0, -\mu_r, \dots, -\mu_1).
$$

For any two fixed partitions $\lambda$, $\mu$, we may consider the sequence of representations corresponding to $[\lambda,\mu]_n$ for large enough $n$ ($n \ge l(\lambda) + l(\mu)$). We call the families of characters we obtain in this way \emph{stable} families of characters of $U(n)$. By \cite[Proposition 2.2]{Koike}, each stable family $([\lambda,\mu]_n)$ defines an element $\chi_{[\lambda,\mu]}$ of $\A$: that is, there is some $\chi_{[\lambda,\mu]} \in \A$ that restricts to the character of $[\lambda,\mu]_n$ on $U(n)$ for large enough $n$. Furthermore, by \cite[Corollary 2.3.2]{Koike}, the $\chi_{[\lambda, \mu]}$ form a linear basis of $\A$. 

\begin{example}
The simplest non-trivial family of characters is defined by $(1,0,...)$ and corresponds to $\xi_1 = \tr$. Another simple family is $(1, 0, ..., 0, -1)$, which corresponds to the irreducible character $\chi(A) = \tr(A)\tr(A^{-1}) - 1$.
\end{example}
\begin{remark}
The algebra $\A$ of stable rational class functions and the notion of stable characters were also discussed in \cite[Section 1.3]{Hanany}.
\end{remark}

Denote by $\widehat{U(\infty)}$ the set of all stable families of irreducible characters, viewed as a subset of $\A$. By orthogonality of irreducible characters, Corollary \ref{cor:Ewf} implies that :
\begin{cor}
If $\chi\in\widehat{U(\infty)}$, $\chi \neq 1, \tr,\overline{\tr}$, and $w\in F_r$ is a non-power, then
$$
\EE_w[\chi] = O(n^{-\pi(w)}).
$$
\end{cor}

The following much stronger bound was conjectured in \cite[Conjecture 1.13]{Hanany} for certain natural families of groups, including unitary and symmetric groups:

\begin{con}\label{con:dimchi}
For any $\chi \in \widehat{U(\infty)}$, and any $w\in F_r$,
$$
\EE_w[\chi] = O\left((\dim\chi)^{1-\pi(w)}\right)
$$
\end{con}

The conjecture holds for $\chi=\tr$ and $\chi=\overline{\tr}$ by the bound $\EE_w[\tr] = O(n^{1-2\cl(w)})$ explained earlier (see \cite[Corollary 1.8]{PuderMageeUn}), since $\pi(w) \le 2\cl(w)$ and $\dim \tr = n$. 
The second simplest character is arguably $\chi(A) = \tr(A)\tr(A^{-1}) - 1$, which is of dimension $n^2-1$ (and corresponds to the adjoint representation of $U(n)$ on $\mathfrak{su}(n)$). Theorem \ref{thm:EwX2} implies that the conjecture holds for this character as well, answering a question of Hanany and Puder in \cite[Section 1.3]{Hanany}:

\begin{cor}
For the character $\chi = \tr(A)\tr(A^{-1}) - 1$, for all $w\in F_r$,
$$
\EE_w[\chi] = O((\dim \chi)^{1-\pi(w)}).
$$
\end{cor}

\begin{remark}
This specific character is also mentioned in \cite[Open problem 6.2]{PuderMageeUn}, where it is asked when it deviates from $1$.
\end{remark}

\begin{remark}
The following interesting observation was made in \cite{Hanany} regarding \emph{polynomial stable characters}. A polynomial stable character is simply a stable character belonging to $\CC[\xi_1,\xi_2,...]\subseteq \A$, or, equivalently, a character whose highest-weight vector is non-negative (so if $[\lambda, \mu]$ is its associated pair of partitions, $\mu$ is the empty partition). Hanany and Puder show that for any $w\in F_r$, Conjecture \ref{con:dimchi} holds for all polynomial stable characters if and only if
$$
\pi(w) \le 2 \scl(w) + 1,
$$
where $\scl(w) := \lim_{k\to\infty}\frac{\cl(w^k)}{k}$ is the \emph{stable commutator length}, another well-known and important algebraic invariant of words. This algebraic inequality was independently conjectured in \cite[Section 6.3.2]{Heuer}, and was verified for various words by computer experiments (see \cite[Proposition 4.4]{CHscl}). Finally, we mention that the argument of Hanany and Puder does not hold for non-polynomial characters like the one in our paper, making it an interesting case. 
\end{remark}

\begin{remark}
So far we discussed some characters for which the conjecture holds. We can also ask for which words the conjecture is known to be true (for any character). For example, the conjecture holds for all words of primitivity rank 1. Indeed, if $\pi(w) = 1$, $w$ is a proper power and the conjecture follows from the power series expansion of \cite{PuderMageeUn}: there are no contributions of order $\Theta(n)$, since they must correspond to a surface with at least one component of genus $0$ and one boundary component (any other surface with boundary has $\chi\le 0$) - such a surface cannot be admissible since the boundary of this component is contractible and hence cannot map to any non-trivial power of $w$.
For $\pi(w) \ge 2$, some specific cases are known. The case of $w=[x,y]$ is well-known: a result of Frobenius (\cite{Frobenius}) states that $\EE_{[x,y]}(\chi) = \frac{1}{\dim\chi}$ for any irreducible character $\chi$, and $\pi([x,y]) = 2$. From this it can be shown (see \cite{Hanany}) that the conjecture holds as well for all words of the form $w=[x_1,y_1]\cdots [x_g, y_g] z_1^{m_1}\cdots z_k^{m_k}$ where the $x_i, y_i, z_j$ are part of a free basis.
\end{remark}

\subsection{Algebraic interpretation of $|\CommCrit(w)|$}\label{subsec:alg-intro}
Suppose $\pi(w) = 2\cl(w)$ and let $g=\cl(w)$. Then $|\CommCrit(w)|$ is the number of equivalence classes of solutions to 
\begin{equation}\label{eq:com_prod}
    [u_1, v_1]\cdots [u_g, v_g] = w
\end{equation}
where two solutions $(u_1,...,u_g,v_1,...,v_g)$ and $(u'_1,...,u'_g,v'_1,...,v'_g)$ are said to be equivalent if they generate the same subgroup, that is,  $\grp{u_1,...,u_g,v_1,...,v_g} = \grp{u_1',...,u_g',v_1',...,v_g'}$. Indeed, clearly every $H\in \CommCrit(w)$ gives rise to a solution of Equation (\ref{eq:com_prod}), and different groups correspond to different solutions. Conversely, any solution $(u_1,...,u_g,v_1,...,v_g)$ of Equation (\ref{eq:com_prod}) generates some subgroup $H$. This group is of rank $\le 2g$. However, it contains $w$ as a non-primitive element (a product of commutators is never primitive), and hence is of rank at least $\pi(w) = 2g$. Hence it is of rank exactly $2g$. Any set of $k$ generators of a free group of rank $k$ is a free basis (see \cite[Proposition 2.7]{LyndonSchupp}), so $\{u_i, v_i\}$ must be a free basis of $H$, hence $H\in\CommCrit(w)$.

The equivalence relation we introduced on the solutions of Equation (\ref{eq:com_prod}) can also be described in another way, following \cite{PuderMageeUn}. Let $F_{2g}$ be the free group with generators $a_1,...,a_g,b_1,...,b_g$. The subgroup $\Aut(F_{2g})_{[a_1,b_1]\cdots [a_g,b_g]}$ of $\Aut(F_{2g})$ of automorphisms which preserve $[a_1,b_1]\cdots [a_g,b_g]$ naturally acts on solutions to Equation (\ref{eq:com_prod}). Any two solutions in the same orbit generate the same subgroup, so they are equivalent by our definition. Conversely, consider a pair of equivalent solutions. They both generate the same subgroup $H$, which is of rank $2g$. Furthermore, each of the solutions gives a natural isomorphism $H\xrightarrow{\sim} F_{2g}$, with $w$ mapping to $[a_1,b_1]\cdots[a_g, b_g]$. Under the isomorphism corresponding to the first solution, the second solution must map to another basis of $F_{2g}$. This basis can be obtained by applying some automoprhism of $F_{2g}$ on the first solution. Furthermore, this automorphism preserves $w$ in $H$, and hence preserves $[a_1,b_1]\cdots[a_g, b_g]$ in $F_{2g}$. Thus any solution can be obtained from any equivalent solution by the action of some element of $\Aut(F_{2g})_{[a_1,b_1]\cdots [a_g,b_g]}$.

\subsection{Similar phenomena in other families of groups}\label{subsec:othergrps}
It is interesting to note that similar results have been proven for other families of groups besides $U(n)$. Most notably, in \cite{Hanany}, Hanany and Puder show that for stable class functions on $S_n$ (their definition is similar to ours):
$$
\EE_w[f] = \grp{f,1} + \grp{f, \rho}\cdot |\Crit(w)| \cdot n^{1-\pi(w)} + O(n^{-\pi(w)}).
$$
where $\rho = \tr - 1$ is the standard character of $S_n$ and $\Crit(w)$ is the set of all subgroups of $F_r$ of rank $\pi(w)$ containing $w$ as a non-primitive element. Observe that in contrast with our case, $\Crit(w)$ is never empty, so in the case of $S_n$ we know exactly what is the second leading term. From this estimate they deduce that for stable irreducible characters $\chi$ of $S_n$ other than $1$ and $\tr - 1$, 
$$
\EE_w[\chi] = O(n^{-\pi(w)}).
$$

In \cite{yotam}, these results are generalized to the case of wreath products $G\wr S_n$, where $G$ is some finite group (continuing a line of work which began in \cite{MPSurfaceWords}). We state their results using a slightly different notation, to emphasize the similarity with the cases of $U(n)$ and $S_n$. We begin by reviewing some necessary preliminaries (see \cite{yotam} for further details and references). Fix a finite group $G$. For the family of groups $G\wr S_n$, one can define a notion of stable class functions and stable families of irreducible characters, similarly to the cases of $S_n$ and $U(n)$. It turns out that any irreducible character $\phi$ of $G$ gives rise to a stable family of irreducible characters $\chi_\phi$ of degree 1 (so $\dim \chi_\phi = \Theta(n)$), and in fact any stable family of irreducible characters of degree $1$ is obtained in this way. For a non-trivial irreducible character $\phi$ of $G$, we say that a f.g. group $H\le F_r$ is a $\phi$-witness of $w$ if it contains $w$ and if the expected value of $\phi(g(w))$ under a choice of a random map $g:H\to G$ (obtained by choosing uniformly the images of the generators) is non-zero. In other words, we consider the word measure induced on $G$ by $w$, when we think of $w$ as an element of $H$ (instead of an element of $F_r$). For the trivial character $\phi = 1$, we say that any group containing $w$ is a $\phi$-witness (this is different from the definition in \cite{yotam}, see the remark below). The $\phi$-witnesses of $w$ are analogous to groups containing $w$ as a surface word, which we study in the case of $U(n)$: to illustrate why, consider the case where $G = \ZZ/m$ is cyclic and $\phi$ is the character $k\mapsto \exp{2\pi i k/m}$. According to \cite[Example 6.6]{yotam}, a group $H$ is a $\phi$-witness of $w$ in this case if and only if $w$ can be written as a word in the basis of $H$ such that the total exponent of each letter is divisible by $m$ (in the language of \cite{yotam} or \cite{MPSurfaceWords}, we should have that $w\in K_m(H)$, where $K_m(H)=\ker{(H\onto H^{ab}/m)}$). Hence we see that for a given $G$ and $\phi$, the $\phi$-witnesses of $w$ are the groups which contain $w$ as a word with a certain combinatorial structure, similarly to groups containing $w$ as a surface word. Furthermore, similarly to the notion of commutator critical subgroups, let $\Crit^\pi_\phi(w)$ be the set of $\phi$-witnesses of $w$ which are algebraic extensions of $w$ and have rank $\pi(w)$. Note that if $\phi$ is the trivial character, we have that $\Crit^\pi_1(w) = \Crit(w)$ in the sense of \cite{Hanany}. For any irreducible character $\phi$ of $G$, we set 
$$
\crit^\pi_\phi(w) = \sum_{H \in \Crit^\pi_\phi(w)} \EE_{\mu_{w\in H}}[\phi]
$$
where for a subgroup $H$, the measure $\mu_{w\in H}$ is the measure $w$ induces on $G$ if we think of $w$ as a word in $H$, as described above. In particular, $\crit^\pi_1(w) = |\Crit(w)|$. Shomroni shows that for any stable class function $f$,  
$$
\EE_w[f] = \grp{f, 1} + \left(\sum_{\phi \in \hat{G}} \grp{f, \chi_\phi}\cdot \crit_\phi^\pi(w) \right) \cdot n^{1 - \pi(w)} + O(n^{-\pi(w)}).
$$
As any stable family of irreducible characters of degree $1$ comes from some $\phi\in\hat{G}$, it follows that for any non-trivial stable family of irreducible characters $\chi$ of $G\wr S_n$,
$$
\EE_w[\chi] = \begin{cases}
\crit^\pi_\phi(w) \cdot n^{1 - \pi(w)} + O(n^{-\pi(w)}), & \chi = \chi_\phi \\
O(n^{-\pi(w)}), & \deg \chi \ge 2
\end{cases}
$$
Note that this is compatible with Conjecture \ref{con:dimchi}, since if $\chi = \chi_\phi$ then $\dim\chi = \Theta(n)$, as mentioned above.

\begin{remark}
We now point out the differences between our statement and the original statement of \cite{yotam}. It is important to remark that the definition of $\phi$-witnesses coincides with the one we give only for non-trivial irreducible characters. Furthermore, in \cite{yotam}, for any non-trivial $\phi\in\hat{G}$ one considers the minimal possible rank $\pi_\phi(w)$ of a $\phi$-witness of $w$, and the set $\Crit_\phi(w)$ of $\phi$-witnesses of this minimal rank. We remark that such witnesses of minimal rank are necessarily algebraic extensions of $w$. One also defines $\crit_\phi(w)$ as $\sum_{H\in \Crit_\phi(w)} \EE_{\mu_{w\to H}}[\phi]$. Finally, denote by $\chi_1$ the family of characters corresponding to the trivial character $\phi = 1$. Then it is shown that
$$
\EE_w[f] = \grp{f, 1} + \grp{f, \chi_1} \cdot |\Crit{w}| \cdot n^{1 - \pi(w)} + \sum_{\phi\in\hat{G} - \{1\}} \grp{f, \chi_\phi} \cdot \crit_\phi(w) \cdot n^{1 - \pi_\phi(w)} + O(n^{-\pi(w)}).
$$
However, when $\pi_\phi(w) > \pi(w)$, the term corresponding to $\phi$ is subsumed by the $O(n^{-\pi(w)})$ term, and indeed in those cases $\Crit^\pi_\phi{w}$ is empty. It is also worth mentioning that Shomroni obtains a stronger bound on $\EE_w[\chi_\phi]$: it is shown that
$$
\EE_w[\chi_\phi] = \crit_\phi(w) \cdot n^{1 - \pi_\phi(w)} + O(n^{-\pi_\phi(w)}).
$$
\end{remark}

Two another interesting families of groups are the orthogonal groups $O(n)$ and the symplectic groups $Sp(n)$. Those cases were studied by Magee and Puder in \cite{MPOn}. In our notation, they show that for each of these families, each moments $\EE_w[\xi_{m_1}\cdots \xi_{m_\ell}]$ admits a power series expansion in terms of certain admissible surfaces, similarly to the case of $U(n)$. It is interesting to see if one can state and prove a result similar to the results discussed in this section for those families of groups.

Finally, one can also consider the family of groups $GL_N(\mathbb{F}_q)$. This case is studied in \cite{glnfq}, where it is conjectured that similar phenomena hold (see e.g. \cite[Conjecture 1.15]{glnfq}), and partial evidence towards this conjecture are given. It is interesting to note that the moments of word measures in this case are rational functions of $q^N$ (and not of $N$), and that instead of algebraic extensions of $w$, the key objects are certain ideals of the group algebra $\mathbb{F}_q[F_r]$.

\subsection{Structure of the Paper}
In Section \ref{sec:pf} we give an overview of the proof. In Section \ref{sec:recap} we review the results of \cite{PuderMageeUn}. In Section \ref{sec:algext} we study the algebraic properties of admissible pairs. In Section \ref{sec:naive} we prove Theorem \ref{thm:EwT}, and in Section \ref{sec:dep} we prove the stronger Theorem \ref{thm:EwTpi}. The proofs are independent of each other. In Section \ref{sec:pfthm2} we give the proof of Theorem \ref{thm:EwX2}.

\subsection{Acknowledgements}
This paper forms an M.Sc. Thesis written by the author under the supervision of Prof. Doron Puder from Tel Aviv University. This project has received funding from the European Research Council (ERC) under the European Union’s Horizon 2020 research and innovation programme (grant agreement No 850956).

\section{Overview of the proof}\label{sec:pf}
Recall that we are trying to bound $\EE_w[\xi_{m_1}\cdots \xi_{m_\ell}]$. We will use the following power series expansion, obtained by Magee and Puder in \cite{PuderMageeUn}:
\begin{equation}\label{eq:mp}
\EE_w[\xi_{m_1}\cdots \xi_{m_\ell}] = \sum_{(\S, f)\; \text{admissible}} \chi^{(2)}(\MCG(f))\cdot n^{\chi(\S)},
\end{equation}
where an admissible pair $(\S, f)$ consists of a surface $\S$ with $\ell$ boundary components and $f$ is a map $f:\S\to \bigvee^r S^1$ that maps the $i$-th boundary component to the loop representing $w^{m_i}$ (exact definitions and further details are reviewed in Section \ref{sec:recap}).

A naive approach would be to bound $\chi(\S)$ for admissible pairs $(\S, f)$. Unfortunately, this is too naive. To see the kind of problem that might arise, consider the case of $\EE_w[\xi_1\xi_{-1}]$ (that is, the expected value of $\tr(w) \tr(w^{-1})$). An admissible pair in this case is a surface $\S$ with two boundary components, mapping to $w$ and $w^{-1}$ under $f$. The easiest example of such a pair is the annulus $\S = S^1 \times I$, with $f$ being the map $f(x,t) = \gamma(x)$ with $\gamma:S^1 \to \bigvee^r S^1$ representing $w$ (so we first collapse $S^1 \times I$ to $S^1$ and then map it to $\bigvee^r S^1$). We see that one boundary of the annulus maps to $w$ and the other to $w^{-1}$ (the boundary circles map to opposite words in $F_r$ because they have opposite orientations). The Euler characteristic of this surface if $0$, which means that its contribution is $O(1)$ (and in fact it can be shown that it contributes exactly $1$). This shows that $\EE_w[\xi_1\xi_{-1}]$ is $O(1)$. We would like to try and get a bound for the next most dominant term. Here we run into a problem: consider the same annulus as before, but attach a handle to it. This is a new admissible surface, of Euler characteristic $-2$. Hence it contributes $Cn^{-2}$. Hence, at least a priori, we get that the second most dominant term is $O(n^{-2})$, which is quite a bad bound. However, this phenomenon is in some sense independent of $w$: we did not use any property of $w$ in constructing this admissible pair. We can phrase this in another way: the map $f:\S \to \bigvee^r S^1$ factors through the map $\gamma:S^1 \to \bigvee^r S^1$ which represents $w$.

To handle this problem, we distinguish between admissible pairs with $f$ factoring through $\gamma$ and those that cannot be factored. For a pair of the former type, write $f$ as $\S \xrightarrow{g} S^1 \to \bigvee^r S^1$. If we consider the pair $(\S, g)$ instead of $(\S, f)$, we can think of it as an admissible pair for $x^{m_1},\dots, x^{m_\ell}$, where $x\in F_1=\pi_1(S^1)$ is a generator. Hence those type of pairs are incorporated in the term $\EE_x[T] = \EE_{unif}[T] = \grp{T, 1}$. We will show that the contribution of those pairs is exactly $\EE_{unif}[T]$, so it remains to bound $\chi(\S)$ only for "interesting" pairs. In fact, we are going to generalize this idea. Let $\Om = \bigvee^r S^1$. We will classify admissible pairs $f:\S\to\Om$ by the image $f_*\pi_1(\S)\le \pi_1(\Om)$. Geometrically, this corresponds to factoring the map $f:\S\to\Om$ as a surjection $\S\to \Gamma$ to a graph $\Gamma$ followed by an immersion $\Gamma \imto \Om$, such that $\pi_1(\Ga) \cong f_*\pi_1(\S)$ and $\pi_1(\Ga) \to \pi_1(\Om)$ is injective (see \cite{stallings} or \cite{stallings-surv} for the relation between graph immersions and subgroups of free groups). We will show that unless $\Ga\imto \Om$ is "complicated enough", the contribution of the pair $(\S, f)$ to the sum in (\ref{eq:mp}) vanishes. More precisely, we will show that if the image of $\pi_1(\Ga)$ in $F_r = \pi_1(\Om)$ (which is just $f_*\pi_1(\S)$ by construction) is not an \emph{algebraic extension} of $\grp{w}$ (that is, $w$ is not contained in a free factor of $f_*\pi_1(\S)$ - see Definition \ref{def:algext}), the contribution of $(\S, f)$ vanishes. It is important to note that this is not entirely correct as stated --- a point we ignore for the moment, but will elaborate on later. It can then be deduced that $(\S, f)$ can only contribute if either $f_*\pi_1(\S) \le \grp{w}$, corresponding to "trivial" cases like the annulus, or $\rk \pi_1( = \rk f_*\pi_1(\S) \ge \pi(w)$ (since the primitivity rank is also the rank of the smallest non-trivial algebraic extension of $w$, see Section \ref{sec:algext}). As a surface with boundary is homotopy equivalent to a graph, we get that if $\S$ is connected (we remark about the non-connected case later), $\chi(\S) = 1 - \rk \pi_1(\S) \le 1 - \rk f_* \pi_1(\S)$ which is bounded by $1 - \pi(w)$ as desired.

The above argument oversees a couple of technical issues. First, we run into a problem if $\S$ is non-connected. We postpone this issue for now. The second critical issue is that the group $f_*\pi_1(\S)$ is hard to work with, as it does not contain all the information about the pair $(\S, f)$. For instance, as the boundaries all have different base points, it does not even contain the boundary words $w^{m_i}$ - only their conjugates. To overcome this, it is convenient to consider instead the relative fundamental groupoid $\Pi_1(\S, V)$, where $V$ is the collection of marked points on the boundary. This is defined to be the category whose objects are the points in $V$ and whose arrows are the homotopy equivalence classes of paths between them. (It is also related to $\pi_1(\S/V)$.) As we only allow homotopies relative to the boundaries, the image of any path between the marked points is a well-defined word in $F_r$. In fact, this groupoid captures exactly the information we want: by Proposition \ref{prop:fstar}, the map $f$ is determined (up to homotopy) by the induced map $f_*:\Pi_1(\S, V) \to \pi_1(\Om)$.

The exact statement we prove is that $(\S, f)$ has vanishing contribution if $\grp{f_*\Pi_1(\S, V)}$,the group generated by the image $f_*\Pi_1(\S, V)$, is not an algebraic extension of $\grp{w}$ (or more precisely of $\grp{w^d}$ for some specific $d$). However, replacing $\pi_1(\S)$ by $\Pi_1(\S, V)$ introduces a further complication: we cannot relate $\chi(\S)$ to $\Pi_1(\S, V)$ in a simple way. A possible solution is to consider the group $\pi_1(\S/V)$ - the fundamental group of the space obtained from $\S$ by gluing the points of $V$ together. Since $f$ maps all the points of $V$ to the same point in $\Om$, the map $f$ descends to a map $\S/V\to \Om$ and we have that $\grp{f_*\Pi_1(\S, V)} = f_*\pi_1(\S/V)$. We can relate $\pi_1(\S/V)$ to $\pi_1(\S)$ quite easily - as it is obtained by gluing $\ell$ points together, we have that 
$$
\rk \pi_1(\S/ V) = \rk \pi_1(\S) + (\ell - 1).
$$
This allows us to relate the group $\grp{f_*\Pi_1(\S,V)}=f_*\pi_1(\S/V)$ to $\rk\pi_1(\S)$ and hence to $\chi(\S)$. However, this bound is not good enough for us. We improve it by using more information about the structure of our problem.

We will present two ways to improve it. The first, which gives a weaker bound, is to note that even though we have a difference of $(\ell - 1)$ in the rank, we also have $\ell$ boundary components which map into the same cyclic subgroup $\grp{w}$. Thus, the rank of $\grp{f_*\Pi_1(\S, V)}$ should be smaller than we expect: $\rk \grp{f_*\Pi_1(\S, V)} = \rk f_*\pi_1(\S/ V) \le \rk \pi_1(\S/V) - (\ell - 1)$. This cancels out the contribution of gluing the points of $V$ together, giving us a total bound of $\rk \pi_1(\S) \ge \rk \grp{f_*\Pi_1(\S, V)} \ge \pi(w)$. This leads to the bound of Theorem \ref{thm:EwT}.

The second method, which achieves a slightly better result, is to use a theorem by Louder and Wilton (\cite{WiltonLouder}) which allows us to relate $\pi_1(\S)$ to a certain graph which basically represents $\Pi_1(\S, V)$.

We now come back to the connectivity of $\S$. Consider for example the case of $T = \xi_1^2 \xi_{-1}$. We can construct certain admissible pairs by taking an annulus with boundaries $w$, $w^{-1}$ and another component which can be any admissible surface for $w$. As the annulus has $\chi = 0$, we see that such pairs contribute as if we were calculating $\EE_w[\tr]$. It turns out that this is the only type of problem that can occur, and it is incorporated in the terms $\grp{T, \tr}$, $\grp{T,\overline{\tr}}$. We note that this issue only arises when proving Theorem \ref{thm:EwTpi} and not the weaker Theorem \ref{thm:EwT}.

Finally, for the case of $\chi(A)=\tr(A)\tr(A^{-1}) - 1$, we are able to get an even better bound by noting that admissible surfaces in this case have two boundary components, corresponding to $w$ and $w^{-1}$, so we can glue an annulus connecting them. This results in a closed surface with a map to the bouquet, which we can analyze more easily as we do not need to take multiple basepoints.

\section{A review of word measures on unitary groups}\label{sec:recap}
In this section, we briefly review the results of \cite{PuderMageeUn} on word measures on the unitary groups $U(n)$. We advise the curious reader to consult the original paper for further details.

We summarize the main results we need from \cite{PuderMageeUn} in the following theorem. The relevant definitions (of admissible pairs, $L^2$-Euler characteristic $\chi^{(2)}$, and admissible pairs definable by a collection a matchings) will be reviewed after the theorem statement.

\begin{thm}[Magee-Puder]\label{thm:recap}
Let $w_1, ..., w_\ell \in F_r$ be non-trivial words. Denote by $\EE[\tr(w_1)\cdots \tr(w_\ell)]$ the integral 
$$
\int_{U(n)^r} \tr(w_1(g_1,...,g_r))\cdots \tr(w_\ell(g_1,...,g_r)) dg_1 \cdots dg_r.
$$
Then $\EE[\tr(w_1)\cdots \tr(w_\ell)]$ is a rational function of $n$ (for large enough $n$), and admits the following power series expansion (for large enough $n$):
\begin{equation}\label{eq:tr}
\EE[\tr(w_1)\cdots \tr(w_\ell)] = \sum_{(\S, f) \textnormal{ admissible}} \chi^{(2)}(\MCG(f)) n^{\chi(\S)}.
\end{equation}
Furthermore, if an admissible pair $(\S, f)$ cannot be defined by a collection of matchings, its contribution is $0$ (that is, $\chi^{(2)}(\MCG(f)) = 0$).
\end{thm}

The sum in (\ref{eq:tr}) is taken over all the equivalence classes of admissible pairs $(\S, f)$ for $w^{m_1}, ..., w^{m_\ell}$. The exact definition is as follows:

\begin{define}
Let $w_1, ..., w_\ell \in F_r$, and let $\Om = \bigvee^r S^1$ be a wedge of circles with a basepoint $o$. Fix an identification $\pi_1(\Om, o) \cong F_r$. An \emph{admissible pair for $(w_1,...,w_\ell)$} is a pair $(\S, f)$, where $\S$ is compact oriented surface $\S$ with $\ell$ boundary components $\partial_1, ..., \partial_\ell$, with at least one boundary component on every connected component, together with a choice of points $v_i \in \partial_i$, and $f$ is a map $f:\S \to \Om$ such that $f(v_i) = o$, and the loop $\partial_i$ (with orientation induced from $\S$), based at $v_i$, is mapped to the class of $w_i$ in $\pi_1(\Om, o)$.
\end{define}

The types of equivalence we allow consist of homeomorphisms of the surface $\S$ and homotopies of the map $f$ relative to the marked points $v_i$. Formally, we define:

\begin{define}
Two admissible pairs $(\S, f)$ and $(\S', f')$ are called \emph{equivalent} if there is an orientation-preserving homeomorphism $F:\S\to \S'$, such that $F(v_i) = v_i'$ (where $v_i$, $v_i'$ are the marked points on the boundaries of $\S$, $\S'$, respectively) and $f'\circ F$ and $f$ are homotopic relative to the points $v_i$.
\end{define}

The coefficients are described in terms of certain stabilizers. Consider the mapping class group $\MCG(\S)$ of $\S$, which is the group of homeomorphisms $\S \to \S$ fixing the boundary pointwise, up to isotopies. It has a natural action on the collection of relative homotopy classes of maps $f:(\S, \{v_i\}) \to (\Om, o)$. We denote the stabilizer of the class of $f$ by $\MCG(f)$.

The $L^2$-Euler characteristic $\chi^{(2)}$ is an analytic analog of the standard Euler characteristic for groups, useful in the case of non-compact classifying spaces. It can take any real value, though it can be shown that it always comes out an integer in our situation. See \cite{Luck2002}, \cite{CheegerGromov} for details on $L^2$-cohomology, and \cite[Section 4]{PuderMageeUn} for the applications to word measures.

It remains to explain what it means for an admissible pair $(\S ,f)$ to be defined by a collection of matchings. We study this in the following subsection.

\subsection{Admissible pairs from collections of matchings}\label{subsec:colmatch}

\begin{figure}[h]
    \centering
    \def\svgwidth{\columnwidth}
    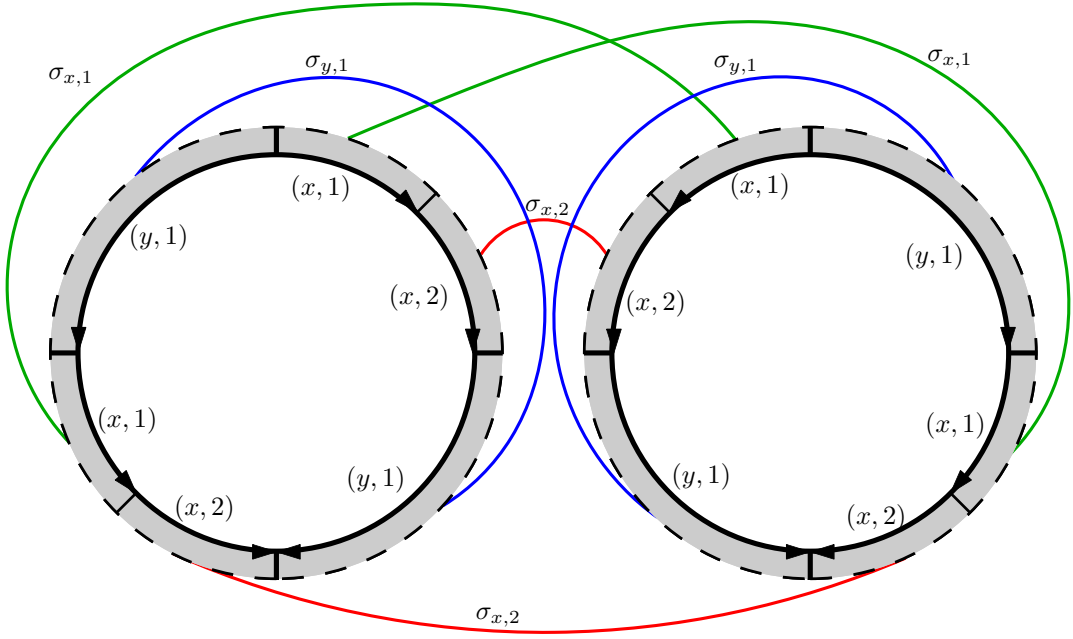
    \caption{An admissible surface for the pair of words $[x,y]$, $[y, x]$. The $x$-segments are divided into two. The annuli should be glued according to the specified matchings.}
    \label{fig:surf-matchings}
\end{figure}

Fix $w_1,...,w_\ell\in F_r$ and fix a basis $x_1,...,x_r$ for $F_r$. Consider the mixed moment
$$\int_{U(n)^r}{\prod_i \tr(w_i(A_1,...,A_r))} \; dA_1 \cdots dA_r.
$$

It is easy to see that this mixed moment vanishes unless the collection $(w_1, \dots, w_\ell)$ is balanced, in the sense that every letter of $F_r$ appears with a total exponent $0$ in all the $w_i$ together (counting $x^{-n}$ as $-n$). Indeed, the map $A_i\mapsto \lambda A_i$ is measure-preserving on $U(n)$ for $|\lambda|=1$. On the other hand, applying it multiplies the integral by $\lambda^{\tau_i}$ where $\tau_i$ is the total exponent of the $i$-th letter. Hence if $\tau_i \neq 0$ the integral must vanish.

Fix a representation for each $w_i$ using the fixed basis $x_1,...,x_r$ (not necessarily reduced). Given a basis letter $x \in \{x_1,...,x_r\}$, we define an $x$-matching of $w_1,...,w_\ell$ to be a perfect matching between the appearances of $x$ and the appearances of $x^{-1}$ in $w_1,...,w_\ell$ (as explained, such a matching must exist for the integral to be non-zero). Given a collection of $x$-matchings for all basis letters (with possibly more than one matching per letter), we will define an admissible pair $(\S, f)$. We say that an admissible pair \textit{can be defined by a collection of matchings} if it is equivalent to an admissible pair obtained by this construction.

We now explain the construction. Let $\sigma_{i,1},\cdots \sigma_{i, k_i}$ be $x_i$-matchings for $w_1,...,w_\ell$ (here $k_i > 0$). We construct an admissible pair $(\S, f)$ as follows. Take $\ell$ annuli, and partition the $m$-th annulus according to the letters of $w_m$. Mark a point on one boundary component of each annulus, at the beginning of the corresponding word $w_i$. We call the boundary component with the marked point the \emph{inner} boundary of the annulus, and the other component the \emph{outer} boundary. Next, for each letter $x_i$, take all the segments corresponding to it (or its inverse) and divide each to $k_i$ segments. The $j$-th segment in each $x_i$ segment is called an $(x_i, j)$-segment. The order is reversed in segments labelled $x_i^{-1}$. Finally, for any $i$ and $j$, glue together $(x_i, j)$-segments on the outer boundaries of all the annuli, according to the matching $\sigma_{i, j}$. This process yields a surface $\S$ with $\ell$ boundary components. See Figure \ref{fig:surf-matchings} for an illustration of the construction. This process also induces a natural map $f:\S \to \Om$: the circles of $\Om$ naturally correspond to $x_1,...,x_r$. Divide the $i$-th circle to $k_i$ segments and label them $(x_i, 1),...,(x_i, k_i)$. On each annulus we can define a map by sending each $(x_i, j)$ segment of the annulus to the corresponding segment of $\Om$. After gluing the annuli together, we obtain a well-defined map $f:\S\to \Om$. Furthermore, $f$ maps each boundary component to the corresponding $w_k$. Hence $(\S, f)$ is an admissible pair.

If $(\S ,f)$ was constructed from a collection of matchings, we can realize the collection of matchings geometrically. Consider the mid-points of the $(x_i, j)$ segments in $\Om$ (that we used for defining $f$). The pre-image $f^{-1}(p)$ of such a point is a collection of arcs (paths between boundary components) on $\S$, which induce the given matching on $(x_i, j)$ segments. Furthermore, if we cut $\S$ along all of those arcs (of all the points), we remain with a collection of disks.

\section{Algebraic properties of admissible pairs}\label{sec:algext}

In this section, we study algebraic properties of admissible pairs $(\S, f)$ for a collection of words $(w_1,\cdots, w_\ell)$. In particular, we show that an admissible pair is determined (up to homotopy) by a certain map of groupoids, and give an algebraic criterion for the vanishing of $\chi^{(2)}(\MCG(f))$. We focus on connected surfaces, but the results can be easily generalized to the non-connected case. We end this section by applying our results to analyze the power series expansion (\ref{eq:tr}).

\subsection{Admissible pairs and fundamental groupoids}

Let $\S$ be a connected surface with boundary, and let $V$ be a choice of one point on each boundary component of $\S$. Let $\Om$ be the rose $\bigvee^r S^1$, and let $o$ be its natural basepoint. We study maps $f:\S \to \Om$ which send the points of $V$ to $o$ (i.e. maps of pairs $(\S, V) \to (\Om, \{o\})$), up to homotopies relative to $V$. We will study what determines the (relative) homotopy class of $f$.

Note that any path between points of $V$ (including loops based at a point of $V$) maps to a closed loop in $\Om$, whose homotopy class in $\pi_1(\Om)$ is unchanged under homotopies of $f$ (relative to $V$). In fact, as we will soon see, the images of these paths determine the homotopy class of $f$.

Formally, we can think of the collection of homotopy classes of paths between points of $V$ as a "relative" fundamental groupoid of $\S$ (which is also called \emph{the fundamental groupoid of $\S$ on $V$}, see \cite{BrownGroupoids}). We denote it by $\Pi_1(\S, V)$. Any homotopy class of maps $f:(\S, V) \to (\Om, o)$ induces a map $f_*:\Pi_1(\S, V) \to \pi_1(\Om, o)$. Our claim is the following:

\begin{prop}\label{prop:fstar}
The homotopy class of $f$ (relative to $V$) is determined by $f_*$.
\end{prop}
\begin{proof}
We can cut $\S$ along a collection of paths between points of $V$, distinct except at their endpoints, yielding a polygon $\Delta$ that represents $\S$ (meaning that $\S$ can be obtained from $\Delta$ by gluing its edges). This is illustrated in Figure \ref{fig:surf-gon}. 

\begin{figure}[h]
    \centering
    \def\svgwidth{\columnwidth}
    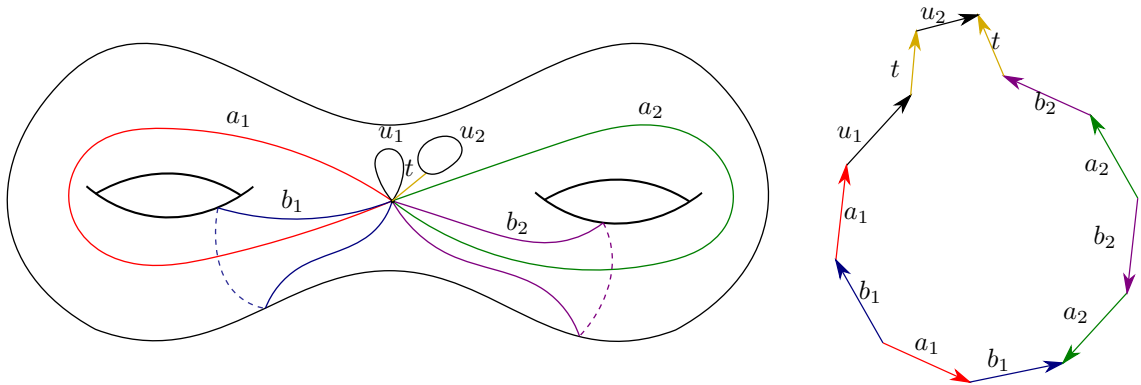
    \caption{Representation of a surface as a polygon. The loops $u_1$ and $u_2$ are the two boundary components}
    \label{fig:surf-gon}
\end{figure}

Any map $f:\S\to\Om$ lifts to a map $\tilde{f}:\Delta\to\Om$. Each boundary edge maps to a loop in $\Om$, and $\partial \Delta$ maps to $1$ in $\pi_1(\Om)$ (since $\Delta$ is contractible). The homotopy classes of the images of the boundary edges depend only on the homotopy class of $f$. Fixing those images, there is a unique way, up to homotopy, to extend the map $\partial\Delta \to \Om$ to the whole of $\Delta$: indeed, note that any two choices of extensions to $\Delta$ yield a map $S^2 \to \Omega$ (where we think of $S^2$ as two copies of $\Delta$ glued to each other). As a graph has no $\pi_2$ such a map must be homotopic to the constant map, showing that the two extensions are homotopic.

Hence, homotopy classes (relative to $V$) of maps $\S \to \Om$ are determined by where they send the edges of $\partial \Delta$, which are elements of $\Pi_1(\S, V)$.
\end{proof}

\begin{remark}
This result holds for non-connected surfaces as well (it is an immediate consequence of the proposition).
\end{remark}

\begin{remark}
The notion of a relative fundamental groupoid $\Pi_1(X, A)$ is related to the notion of a relative fundamental group $\pi_1(X, A) := \pi_1(X/A)$. The map of pairs $f:(\S, V) \to (\Om, \{o\})$ induces maps from both $\Pi_1(\S, V)$ and $\pi_1(\S, V)$ to $F_r \cong \pi_1(\Om, o)$. We have that $\grp{f_*\Pi_1(\S,V)} = f_*\pi_1(\S, V)$, since both groups are generated by the images of paths between points of $V$ under $f_*$.
\end{remark}

\begin{remark}
Using the representation of a surface (with boundary) as a disk with edges glued, one can show that in fact homotopy classes of maps $f:\S\to\Om$ are in correspondence with solutions to 
$$
[x_1, y_1] \cdots [x_g, y_g] u_1 (v_2 u_2 v_2^{-1}) (v_3 u_3 v_3 ^{-1}) \cdots (v_\ell u_\ell v_\ell^{-1}) = 1
$$
in $\pi_1(\Om)$. Indeed, this equation represents the boundary word $\Delta$.
\end{remark}

\subsection{Admissible pairs and algebraic extensions}

\begin{define}\label{def:algext}
Let $J$ be a free group and $H \le J$ be a subgroup. The group $J$ is called an \emph{algebraic extension} of $H$ if $H$ is not contained in any proper free factor of $J$.
\end{define}

Recall that for a free group $F$ and a word $w\in F$, the primitivity rank of $w$, $\pi(w)$, is defined to be the least possible rank of a subgroup $H \le F$ containing $w$ as a non-primitive element. In terms of algebraic extensions, $\pi(w)$ is the minimal rank of a proper algebraic extension of $\grp{w}$. Our main result in this section is the following:

\begin{thm}\label{thm:needalg}
Let $(\S, f)$ be an admissible surface for $(w_1,..., w_\ell)$. Suppose $\S$ is connected. Let $H$ be the group generated by the image of $\Pi_1(\S, V)$ in $F_r$. If $H$ is not an algebraic extension of $\grp{w_1,...,w_\ell}$, then $\chi^{(2)}(\MCG(f)) = 0$.
\end{thm}

The idea behind Theorem \ref{thm:needalg} is that $\chi^{(2)}(\MCG(f))$ can be non-zero only if $(\S, f)$ can be defined by a collection of matchings of $w_1, ..., w_\ell$. In such a case, the image of $f_*$ can only consist of words in the letters appearing in $w_1, ..., w_\ell$. If $\grp{w_1,...,w_\ell}$ is contained in a free factor of $H$, then we could have chosen a basis for $H$ with some letter $u$ not appearing in any $w_i$. Then $u$ cannot be generated by the image of $f_*$, contradicting that $u\in H = \grp{f_*\Pi_1(\S, V)}$. We now explain the details.

\begin{lemma}\label{lem:mapsto}
Let $(\S, f)$ be an admissible surface for $(w_1, ..., w_\ell)$, with $\S$ connected. In particular, $f$ is a map $\S \to \Om$ where $\Om$ is a bouquet with $\pi_1(\Om) \cong F_r$. Let $H = \grp{f_*\Pi_1(\S, V)}$. Since $H$ is a f.g. free group, we can take a bouquet $\Om'$ (with basepoint $o_{\Om'}$) with $\pi_1(\Om') \cong H$. Let $p:\Om' \to \Om$ be the map corresponding to the inclusion $H\into \pi_1(\Om)$. Then there is a map $g:\S \to \Om'$ such that $f:\S \to \Om$ is homotopic to $p\circ g$. Furthermore, the stabilizers $\MCG(f)$ and $\MCG(g)$ are equal (where the latter is defined with respect to the action of $\MCG(\S)$ on maps $(\S, V) \to (\Om', o_{\Om'})$).
\end{lemma}

This lemma allows us to replace $\Om$ by $\Om'$. By this we mean that the map $g:\S \to \Om'$ defines an admissible pair $(\S, g)$ for $(w_1,\dots, w_\ell)$ where we think of them as words in $H = \pi_1(\Om')$. As the corresponding stabilizers are equal by the lemma, it is enough to study $(\S, g)$ instead of $(\S, f)$. The advantage is that by construction, $H = \pi_1(\Om')$, so $g_*$ is surjective. In other words, we may now assume that $f_*$ is surjective. We can now prove Theorem \ref{thm:needalg}.

\begin{proof}[Proof of Theorem \ref{thm:needalg}, assuming Lemma \ref{lem:mapsto}]
If $H$ is not an algebraic extension of $W = \grp{w_1, ..., w_\ell}$, there is a free factor $J \le H$ containing $W$. Thus, we could have chosen a basis of $H$ containing an element $u$ which is independent of $J$. Using Theorem \ref{thm:recap}, it is enough to show that $(\S, f)$ cannot be defined by a collection of matchings (when the words are written in our new chosen basis). However, if $(\S, f)$ is defined by a collection of matchings it is evident that $\grp{f_*\Pi_1(\S, V)}$ consists only of words involving the letters of $w_1,\dots, w_\ell$. By our choice of basis for $H$, the letter $u$ cannot be contained in $\grp{f_*\Pi_1(\S, V)}$, contradicting surjectivity.
\end{proof}

\begin{proof}[Proof of Lemma \ref{lem:mapsto}]
Let $\Delta$ be a polygon representation of $\S$. The map $f:\S\to \Om$ lifts to a map $\tilde{f}:\Delta \to \Om$ (which is simply the composition of $f$ with the natural projection $\Delta \to \S$). By construction, the edges of $\Delta$ map to elements of $H$, so we can lift the map $\tilde{f}|_{\partial\Delta}:\partial\Delta \to \Om$ to a map $\partial\Delta \to \Om'$ (perhaps after an homotopy of $\tilde{f}$):
$$
\xymatrix{\partial \Delta \ar[d]\ar@{-->}[r] \ar[rd]^{\tilde{f}|_{\partial\Delta}} &\Om' \ar[d]^{p}\\ 
           \Delta \ar[r]^{\tilde{f}} &\Om    }
$$
The map $\tilde{f}$ on the interior of $\Delta$ gives a homotopy showing that the image of $\partial\Delta$ in $\pi_1(\Om)$ is trivial. Since the image of $\partial \Delta$ is also trivial in $\pi_1(\Om')$ (since the edges map to elements of $H$ and their product is still trivial in $H$), there is a well-defined map $\tilde{g}:\Delta \to \Om'$. We claim that the map $p\circ \tilde{g}:\Delta \to \Om$ is homotopic to $\tilde{f}$.
$$
\xymatrix{ & \Om'\ar[d]^{p} \\ \Delta \ar[ru]^{\tilde{g}}\ar[r]^{\tilde{f}}& \Om}
$$
To see why, we consider two copies of $\Delta$, glued together at their boundaries. This forms a sphere $S$, and we can map it to $\Om$ by mapping one copy of $\Delta$ by $\tilde{f}$ and the other by $p\circ \tilde{g}$. Since $\pi_2(\Om)$ is trivial, this map $S\to\Om$ is contractible, and hence can be lifted to a map $D\to \Om$ where $D$ is a ball with boundary $S$. The two copies of $\Delta$ in $S$ are homotopic in $D$, and the homotopy can be chosen to fix their boundaries, and hence the maps $\tilde{f}$ and $p\circ \tilde{g}$ are homotopic and moreover the homotopy between them can be chosen to be constant on $\partial\Delta$. Now, the map $\tilde{g}$ descends to a map $g:\S \to \Om'$ (since $\tilde{f}$ descends to a map $f$ on $\S$ and they agree on $\partial\Delta$), and so $f \sim p \circ g$. Finally, the mapping class group of $\S$ acts on homotopy classes of maps $\S \to \Om'$ (relative to $V$). The stabilizer of $g$ is clearly contained in the stabilizer of $f$. For the other direction, it is enough to note that if two maps $\S\to \Om' \to \Om$ are homotopic as maps to $\Om$, we can lift the homotopy to $\Om'$. This is clear - considering the polygon representation of $\S$, we see that both maps must have the same elements of $F_r$ on the edges, and furthermore those elements are in $H$ (since they factor through $\Om'$). In particular it is clear that the maps are homotopic as maps into $\Om'$.
\end{proof}

In our case, since $w_i = w^{m_i}$, the group $\grp{w_1,...,w_\ell}$ is simply $\grp{w^d}$ for $d=\gcd(m_1,...,m_\ell)$. Hence it is important to understand algebraic extensions of $\grp{w^d}$. 

\begin{lemma}\label{lem:lemma2}
Let $w$ be a non-power and let $H$ be an algebraic extension of $\grp{w^d}$. Then $\grp{H,w}$ is an algebraic extension of $w$, and if $H \neq \grp{w^i}$ (for any $i$), then $\rk(H) \ge \pi(w)$.
\end{lemma}
\begin{proof}
This is proved in the proof of \cite[Lemma 7.5]{Hanany}, but not stated explicitly. We give a full argument for the sake of completeness. Consider $H' = \grp{H, w}$. We claim that $H'$ is an algebraic extension of $w$. This follows from the following lemma (\cite[Lemma 7.4]{Hanany}, \cite[Corollary 3.14]{AlgExts}): Let $J\le F_r$ be a finitely generated subgroup and let $u\in F_r$. If $\rk J \ge \rk\grp{J, u}$, then $\grp{J,u}$ is an algebraic extension of $J$.

To apply this lemma, we need to show that $\rk H' \le \rk H$. Clearly, $\rk H' \le \rk H + 1$. However, we cannot have $\rk H' = \rk H + 1$, since it would imply that any generating set of size $\rk H + 1$ is free (a well-known property of free groups), and we have a generating set of this size which is not free: namely, a generating set for $H$ together with $w$ (which satisfies an equation because $w^d \in H$). Hence the lemma applies, and $\grp{H, w}$ is an algebraic extension of $H$, which is an algebraic extension of $\grp{w^d}$. The relation of algebraic extensions is transitive (\cite[Proposition 3.11]{AlgExts}, \cite[Claim 4.1]{PuderParWordMeasures}), hence we get that $\grp{H, w}$ is an algebraic extension of $\grp{w^d}$, which implies it must be an algebraic extension of $\grp{w}$ as well.

If $H' = \grp{w}$, then $H$ is generated by a power of $w$. Otherwise, $\rk(H') \ge \pi(w)$. On the other hand, $\rk(H) \ge \rk(H')$, as explained before, so $\rk H \ge \pi(w)$. 
\end{proof}

The following observation will be useful for us later:

\begin{claim}\label{clm:Hw}
Let $w$ be a non-power and let $H$ be a f.g. subgroup of $F_r$. If $\grp{H, w}$ is an algebraic extension of $w$ and $H\neq \grp{w^i}$ (for any $i$), then $\rk H \ge \pi(w)$.
\end{claim}
\begin{proof}
Let $H' = \grp{H,w}$. We either have $\rk H' = \rk H$ or $\rk H' = \rk H + 1$. The latter is impossible: if $\rk H' = \rk H + 1$, then any basis of $H$ becomes a basis of $H'$ when adding $w$ (since it is a generating set of size $\rk H'$), and hence $w$ is primitive in $H'$. Hence we must have $\rk H = \rk H'$, but $\rk H' \ge \pi(w)$.
\end{proof}

\subsection{A certain family of admissible surfaces}
Recall the notation of Theorem \ref{thm:EwTpi}. We now use Theorem \ref{thm:needalg} to show that in computing $\EE_w[T]$, we can restrict to a certain class of admissible surfaces. We then use this to explain how we shall prove the main theorems of the paper.

By Theorem \ref{thm:recap},
$$
\EE_w[T](n) = \sum_{(\S, f) \textnormal{ admissible}} \chi^{(2)}(\MCG(f)) n^{\chi(\S)}.
$$
Let $(\S, f)$ be an admissible pair for $(w^{m_1}, ..., w^{m_\ell})$. Let $\S_1,...,\S_k$ be the connected components of $\S$, and let $f_i=f|_{\S_i}$ and $V_i=V\cap \S_i$ (the marked points). It is clear that $\MCG(\S) = \Pi_i \MCG(\S_i)$, and so $\MCG(f) = \Pi_i \MCG(f_i)$. Hence, by \cite[Formula 7.14]{Luck2002}, 
$$
\chi^{(2)}(\MCG(f)) = \prod_i \chi^{(2)}(\MCG(f_i)).
$$
In particular, for the pair $(\S, f)$ to give a nonzero contribution, we must have $\chi^{(2)}(\MCG(f_i))\neq 0$ for all $i$. However, note that each $(\S_i, f_i)$ is an admissible pair for some collection of powers of $w$, and so by Theorem \ref{thm:needalg} we must have that that $H_i:=\grp{{f_i}_*\Pi_1(\S_i, V_i)}$ is an algebraic extension of some power of $w$ (which is the $\gcd$ of the exponents of the boundary words). By Lemma \ref{lem:lemma2}, we must have that $\grp{H_i, w}$ is an algebraic extension of $w$, and we either have $H_i \le \grp{w}$ or $\rk H_i \ge \pi(w)$.

Consider an admissible pair $(\S, f)$ such that for any component $\S_i$ of $\S$, $\grp{f_*\Pi_1(\S_i, V_i)} \le \grp{w}$. By Lemma \ref{lem:mapsto} (and its proof), the map $f_i:\S_i \to \Om$ factors through a map $g:\S \to \Om'$ where $\Om'=S^1$ is simply a circle and $\Om'\to \Om$ is the map representing $w\in\pi_1(\Om)$. Since this happens for each component of $\S$, we have that the map $f:\S \to \Om$ itself factors through the map $S^1 \to \Om$. Conversely, if $f$ factors through $S^1 \to \Om$, then clearly $\grp{f_*\Pi_1(\S_i, V_i)}\le \grp{w}$. Moreover, we can extend any map $g:\S\to S^1$ such that the boundaries cover the circle with degrees $m_1,..,m_\ell$ to a map $f:\S \to S^1 \to \Om$ which will give an admissible pair for $(w^{m_1},...,w^{m_\ell})$. Hence we see that there is a bijection between admissible pairs for $(x^{m_1},...,x^{m_\ell})$ (where $x\in F_1$ is a generator) to admissible pairs $(\S, f)$ for $(w^{m_1},...,w^{m_\ell})$ such that for any component $\S_i$ of $\S$, $\grp{f_*\Pi_1(\S_i, V_i)} \le \grp{w}$, and furthermore by Lemma \ref{lem:mapsto} their contributions are the same. But the total contribution of admissible pairs for $(x^{m_1},...,x^{m_\ell})$ is just $\EE_x[T] = \grp{T, 1}$. We arrive to the following corollary:
\begin{cor}\label{cor:EwT1}
Let $w\in F_r$ be a non-power, let $m_i\in\ZZ$ be non-zero integers, let $T = \xi_{m_1}\cdots \xi_{m_\ell}$, and let $d = \gcd(m_1,...,m_\ell)$. Then
$$
\EE_w[T] = \grp{T, 1} + \sum_{(\S, f)\in \mathcal{S}(w^{m_1},...,w^{m_\ell})} \chi^{(2)}(\MCG(f)) n^{\chi(\S)}
$$
where $\mathcal{S}(w^{m_1},...,w^{m_\ell})$ is the set of all equivalence classes of admissible surface $(\S, f)$ such that for all connected components $\S_i$ of $\S$, the group $H_i = \grp{f_*\Pi_1(\S_i, V\cap S_i)}$ is an algebraic extension of $w^k$ (where $k$ is the $\gcd$ of the $m_j$-s corresponding to the circles $\partial\S_i$), and for at least one $\S_i$ the group $H_i$ is non-cyclic.

Furthermore, for any $(\S, f) \in \mathcal{S}(w^{m_1},...,w^{m_\ell})$ there is at least one component $\S_i$ of $\S$ such that $\rk \grp{f_*\Pi_1(\S_i, V\cap \S_i)} \ge \pi(w)$.
\end{cor}

Recall that we want to prove Theorems \ref{thm:EwT} and \ref{thm:EwTpi}: the weaker theorem states that 
$$
\EE_w[T] = \grp{T, 1} + O(n^{1-\pi(w)})
$$
and the stronger one states that
$$
\EE_w[T] = \grp{T,1} + (\grp{T,\tr}+\grp{T,\overline{\tr}})\cdot C(w) \cdot n^{1-2\cl{w}} + O(n^{-\pi(w)}).
$$
In both cases, we aim to bound $\EE_w[T] - \grp{T, 1}$, and so Corollary \ref{cor:EwT1} plays a key role: it allows us to restrict our attention to admissible pairs $(\S, f) \in \mathcal{S}(w^{m_1},...,w^{m_\ell})$. For example, consider such a pair $(\S, f)$ and let $\S_1,...,\S_k$ be the connected components of $\S$. Note that $\chi(\S) = \sum \chi(\S_i)$ and that each $\S_i$ is a surface with boundary which is not a disk, so $\chi(\S_i) \le 0$. Hence if we can prove that for at least one of the components we have $\chi(\S_i) \le 1 - \pi(w)$, we would prove Theorem \ref{thm:EwT}. To this end, note that as a connected surface with boundary, each $\S_i$ is homotopy equivalent to a graph, and so $\chi(\S_i) = 1 - \rk \pi_1(\S_i)$. Hence it is enough to prove that for at least one component $\S_i$, $\rk \pi_1(\S_i) \ge \pi(w)$. The natural candidate is of course the component for which $\grp{f_*\Pi_1(\S_i, V_i)}$ is a non-cyclic algebraic extension of some power of $w^k$. This is exactly what is proved in Proposition \ref{prop:Scon}.

If we try to improve this bound, the following problem arises: in the cases where $\pi(w)=2\cl(w)$, there may be a connected admissible surface for $w$ of genus $\cl(w)$, so its $\chi$ is $1-2\cl(w) = 1-\pi(w)$. Furthermore, we can take a disjoint union of it with surfaces of $\chi=0$ (which must be annuli), giving us possibly nontrivial contributions for the $n^{1-\pi(w)}$ term. We prove that these are the only possible contributions for the coefficient of $n^{1-\pi(w)}$, which gives us Theorem \ref{thm:EwTpi} (after some analysis of this coefficient).

\section{The naive method}\label{sec:naive}

In this section, we prove Theorem \ref{thm:EwT}. We emphasize that the proof is independent the proof of Theorem \ref{thm:EwTpi}, and so the interested reader may skip directly to Section \ref{sec:dep}. Our proof of Theorem \ref{thm:EwT} is based on the following proposition:
\begin{prop}\label{prop:Scon}
Let $(\S, f)$ be an admissible pair for $w^{m_1},\dots, w^{m_k}$ for some $m_1,...,m_\ell \in \ZZ$, not all zero. Let $d = \gcd(m_1,...,m_k)$. Suppose that $\S$ is connected and that $\grp{f_*\Pi_1(\S, V)}$ is a non-cyclic algebraic extension of $w^d$. Then $\rk \pi_1(\S) \ge \pi(w)$.
\end{prop}

As outlined in the end of the previous subsection, Theorem \ref{thm:EwT} follows:
\begin{proof}[Proof of Theorem \ref{thm:EwT}]
By Corollary \ref{cor:EwT1}, it is enough to show that $\chi(\S) \le 1 - \pi(w)$ for any admissible pair $(\S, f)\in\mathcal{S}(w^{m_1},...,w^{m_\ell})$. Fix such a pair $(\S, f)$ and let $\S_1,...,\S_k$ be its connected components. Let $\S_{i_0}$ be a connected component such that $\grp{f_*\Pi_1(\S_{i_0}, V\cap \S_{i_0})}$ is a non-cyclic algebraic extension of $w^k$ (where $k$ is the $\gcd$ of the $m_j$-s corresponding to the boundaries of $\S_{i_0}$), which exists by the definition of $\mathcal{S}$. By Proposition \ref{prop:Scon}, we must have $\rk\pi_1(\S_{i_0}) \ge \pi(w)$. On the other hand, $\S_{i_0}$ is a connected surface with boundary and hence is homotopic to a graph. Thus $\chi(\S_{i_0}) = 1 - \rk\pi_1(\S_{i_0}) \le 1 - \pi(w)$. The other $\S_i$ have $\chi(\S_i) \le 0$ since they are surfaces with boundary, and so we see that $\chi(\S) = \sum \chi(\S_i) \le 1 - \pi(w)$. 
\end{proof}

The rest of this section is devoted to the proof of Proposition \ref{prop:Scon}. The proof follows an algebraic approach.

Let $(\S, f)$ be as above. Denote by $g$ the genus of $\S$ and by $\ell$ its number of boundary components. As usual, $V$ denotes a chosen collection of one point on each boundary component.

We wish to relate $\pi_1(\S)$ to $\Pi_1(\S, V)$. To do so, we consider the group $\pi_1(\S/V)$. Observe that $f_*\pi_1(\S/V) = \grp{f_*\Pi_1(\S, V)}$, as both groups are generated by the images under $f$ of paths between the points of $V$. Hence $f_*\pi_1(\S/V)$ is a non-cyclic algebraic extension of $w^d$. On the other hand, $\S/V$ is obtained by gluing the $\ell$ points of $V$ together, which is homotopy equivalent to wedging $\S$ with $\ell-1$ circles. Hence $\pi_1(\S/V)\cong \pi_1(\S)*F_{\ell-1}$. Furthermore, as $\S$ is a connected surface with boundary it is homotopy equivalent to a graph, and hence its $\pi_1$ is free. Thus $\pi_1(\S/V)$ is a free group as well, with $\rk \pi_1(\S/ V) = \rk \pi_1(\S) + (\ell - 1)$, or, equivalently, $\rk \pi_1(\S) = \rk \pi_1(\S/V) - (\ell - 1)$.

Thus, it is enough to bound $\rk\pi_1(\S/V)$ from below. Since $H = f_*\pi_1(\S/V)$ is a non-cyclic algebraic extension of $w^d$, $\rk{H} \ge \pi(w)$ (by Lemma \ref{lem:lemma2}). As $\rk \pi_1(\S/V) \ge \rk f_*\pi_1(\S/V)$, we conclude that $\rk \pi_1(\S/V) \ge \pi(w)$. This implies that $\rk \pi_1(\S) \ge \pi(w) - (\ell - 1)$, which is not good enough for us.

We can get a better bound by bounding $\rk \pi_1(\S/ V) - \rk H$ from below. Indeed,
\begin{eqnarray*}
   \rk \pi_1(\S) = \rk \pi_1(\S/V) - (\ell - 1) = \rk H + (\rk \pi_1(\S/V) - \rk H) - (\ell - 1) \ge  \\ 
   \ge \pi(w) - (\ell - 1) + (\rk \pi_1(\S/ V) - \rk H).    
\end{eqnarray*}
In particular, to prove Theorem \ref{thm:EwT}, it is enough to show $\rk \pi_1(\S/ V) - \rk H \ge \ell - 1$.

Fix a point $v_1\in V$. Let $t_2,...,t_{\ell}$ be non-intersecting paths connecting the other points of $V$ to $v_1$ (we can also define $t_1$ to be the trivial path from $v_1$ to itself), and let $u_1, ..., u_\ell$ be loops homotopic to the boundary component, each based in the corresponding point of $V$. We can think of both the $t_i$ and the $u_i$ as elements of $\pi_1(\S/V)$, but we also have that for each $i$, $t_i u_i t_i^{-1}$ represents a loop, based at $v_1$, around the $i$-th boundary component. Consider the loop $u_1 t_2 u_2 t_2^{-1} \cdots t_\ell u_\ell t_\ell^{-1}$. It is a loop based at $v_1$ and enclosing all the boundary components. Change it slightly by a homotopy so it does not intersect the boundary other than at $v_1$. It then separates $\S$ to two components - one which is a disk with $\ell$ smaller disks removed from it, and one which is a connected surface of genus $g$ with one disk removed from it. The latter component has a fundamental group which is a free group of rank $2g$, and for a suitable choice of generators $a_1,...,a_g,b_1,...,b_g$ we can write its boundary as $[a_1, b_1]\cdots [a_g, b_g]$. Observe that this boundary is homotopic to our loop $u_1 t_2 u_2 t_2^{-1} \cdots t_\ell u_\ell t_\ell^{-1}$. Looking at $\S/V$, our loop also separates it to two components. The surface component is the same, with the only change being in the disk component: its $\pi_1$ can now be identified with the free group generated by $u_1,...,u_\ell, t_2,...,t_\ell$. Hence, by the van Kampen theorem, we get that $\pi_1(\S/V)$ is the group generated by $a_1,...,a_g,b_1,...,b_g, u_1,...,u_\ell, t_2,...,t_\ell$ subject to the unique relation
$$
[a_1, b_1]\cdots [a_g, b_g] = u_1 t_2 u_2 t_2^{-1} \cdots t_\ell u_\ell t_\ell^{-1}.
$$

Consequently, $H=f_*\pi_1(\S/V)$ is generated by the images of $a_i, b_i, t_j, u_j$ under $f_*$. As all the $u_i$ are mapped to powers of $w$, we see that $H$ is in fact generated by $f_*(a_i)$, $f_*(b_i)$, $f_*(t_i)$, and some $w^d$.

The key observation is that this is $(\ell - 1)$ generators less than the number of generators of $\pi_1(\S/V)$, which yields the desired bound $\rk \pi_1(\S/V) - \rk H \ge \ell - 1$.

Formally, we argue as follows. We see that $H$ is generated by $2g + (\ell - 1) + 1 = 2g + \ell$ generators. However, they satisfy an equation (the image of the equation satisfied in $\pi_1(\S/V)$). As $H$ is free, this implies that $\rk H$ is strictly less than $2g + \ell$. Indeed, it is well known that if $F$ is a free group of rank $r$, any set of $r$ generators is free (see \cite[Proposition 2.7]{LyndonSchupp}). 

Hence $\rk H < 2g + \ell$. On the other hand, $\rk \pi_1(\S/ V) = 2g + 2\ell - 2$. This can be seen either from the presentation, or from the following computation:
\begin{eqnarray*}
\rk \pi_1(\S/V) = \rk \pi_1(\S) + \ell - 1 = 1 - \chi(\S) + \ell - 1 = \\ 
= 1 - (2 - 2g - \ell) + \ell - 1 = 2g + 2\ell - 2
\end{eqnarray*}
Here we used the fact that $\chi(\S) = 2 - 2g - \ell$, as it is a surface of genus $g$ with $\ell$ disks removed.

Concluding, we have that $\rk \pi_1(\S/V) = 2g + 2\ell - 2$ and $\rk H \le 2g + \ell - 1$. Thus
$$
\rk \pi_1(\S/V) - \rk H \ge 2g + 2\ell - 2 - (2g + \ell - 1) = \ell - 1
$$
as desired.

\section{A stronger result}\label{sec:dep}
In this section, we prove our main result (Theorem \ref{thm:EwTpi}): if $w\in F_r$ is a non-power, then for any nonzero $m_1, \dots, m_\ell \in \ZZ$, if we set $T = \xi_{m_1}\dots \xi_{m_\ell}$, we have:
$$
\EE_w[T] = \grp{T,1} + \grp{T, \tr + \overline{\tr}} \cdot |\CommCrit(w)| \cdot n^{1 - \pi(w)} + O(n^{-\pi(w)}).
$$

By Corollary \ref{cor:EwT1}, it is enough to restrict ourselves to admissible pairs $(\S, f)$ such that for each of their connected components $\S_i$, the group $\grp{f_*\Pi_1(\S_i, V\cap \S_i)}$ is an algebraic extension of $w^d$ (where $d$ is the $\gcd$ of the $m_j$-s corresponding to the boundary circles of $\S_i$), and for at least one component this group is not cyclic. As in the previous section, we focus on the "special" connected component: we begin by proving a result on admissible pairs $(\S, f)$ for $(w^{m_1},...,w^{m_\ell})$ where $\S$ is \emph{connected} and $\grp{f_*\Pi_1(\S, V)}$ is a non-cyclic algebraic extension of $w^d$ (for $d=\gcd(m_1,...,m_\ell)$). In fact, it is more convenient to restrict ourselves to a slightly larger class - firstly, we only require that $\grp{f_*\Pi_1(\S, V), w}$ is a non-cyclic algebraic extension of $\grp{w}$ (this is indeed a weaker condition, by Lemma \ref{lem:lemma2}, and by the obvious fact that if a group $H$ is non-cyclic then so is the larger group $\grp{H,w}$). Furthermore, while we usually restrict ourselves to admissible pairs where the boundary words $f_*(\partial_i)$ are non-trivial (that is, $m_i\neq 0$), we will drop this requirement in this section. For this class of surfaces, we show:
\begin{prop}\label{prop:chiS}
Let $m_1,\dots, m_\ell\in \ZZ$, which may be $0$, and let $(\S, f)$ be an admissible pair for $w^{m_1},\dots, w^{m_\ell}$. Suppose that $\S$ is connected and that $\grp{f_*\Pi_1(\S, V),w}$ is a non-cyclic algebraic extension of $\grp{w}$. Then one of the following holds: 
\begin{enumerate}
    \item $\chi(\S) < 1 - \pi(w)$.
    \item $\chi(\S) = 1 -2\cl(w) = 1 - \pi(w)$ and $(\S, f)$ is an admissible pair for either $w$ or $w^{-1}$ (so it is a surface with exactly one boundary component, mapped to either $w$ or $w^{-1}$ by $f_*$.)
    \item $\chi(\S) = 1 - \pi(w)$ and $m_1 = \dots = m_\ell = 0$.
\end{enumerate}
\end{prop}
\begin{remark}\label{rmk:clpi}
In particular, we always have that $\chi(\S) \le 1 - \pi(w)$.
\end{remark}
\begin{remark}
Note that we already know that an admissible pair $(\S, f)$ for $w$ must have $\chi(\S) < 1 - \pi(w)$ unless $\S$ is of genus $\cl(w)$ and $\pi(w)=2\cl(w)$. Indeed, we must have $\chi(\S) \le 1 - 2\cl(w)$ (as mentioned in the introduction, see also \cite[Proposition 1.1]{Culler}), and $\pi(w) \le 2 \cl(w)$ (by Claim \ref{prop:picl}).
\end{remark}
Using this proposition, together with Corollary \ref{cor:EwT1}, it follows that if $T=\xi_{m_1}\cdots \xi_{m_\ell}$, where the $m_i$ are nonzero, there is some constant $C(T, w)$ so that
$$
\EE_w[T] = \grp{T,1} + C(T, w)\cdot n^{1-\pi(w)} + O(n^{-\pi(w)}).
$$
Indeed, consider some $(\S, f) \in \mathcal{S}(w^{m_1},...,w^{m_\ell})$. Let $\S_1, ..., \S_k$ be the connected components of $\S$. Let $\S_{i_0}$ be a component for which $\grp{f_*\Pi_1(\S_{i_0}, V\cap S_{i_0})}$ is a non-cyclic algebraic extension of some $w^k$. Then by Proposition \ref{prop:chiS}, either $(\S_0, f|_{\S_0})$ has $\chi(\S_0) < 1 - \pi(w)$, or $(\S_0, f|_{\S_0})$ is an admissible pair for $w^{\pm 1}$ and $\chi(\S_0) = 1 - 2\cl(w)$ and $\pi(w) = 2\cl(w)$ (the third case is not possible since we assumed the $m_i$ are nonzero). Furthermore, any other connected component is a surface with boundary which is not a disk and hence has $\chi(\S_i) \le 0$. Hence in the first case we trivially obtain that $\chi(\S) = \sum \chi(\S_i) < 1 - \pi(w)$ so the contribution of $(\S, f)$ is $O(n^{-\pi(w)})$. In the second case, $(\S, f)$ contributes $O(n^{1-\pi(w)})$. In fact, the pairs contributing to the $n^{1-\pi(w)}$ are exactly the pairs $(\S, f)$ where $\S$ has one component $\S_0$ such that $(\S_0, f|_{\S_0})$ is admissible for $w^{\pm 1}$, $\chi(\S_0) = 1 - \pi(w)$, and all other components have $\chi = 0$. The coefficient $C(T,w)$ is comprised of the coefficients of those cases.

In Subsection \ref{subsec:coeff}, we use the extra structure we have on pairs with $\chi(\S) = n^{1-\pi(w)}$ to show that
$$
C(T, w) = \grp{T,\tr+\overline{\tr}} \cdot |\CommCrit(w)|.
$$

The proof of Proposition \ref{prop:chiS} relies on a theorem of Louder and Wilton. Before we state it, we introduce the notion of a combinatorial circle.

\begin{define}
A combinatorial circle is a graph which is topologically a circle.
\end{define}

\begin{thm}[{Dependence Theorem, \cite[Theorem 2.21]{WiltonLouder}}]\label{thm:dep}
Let $S$ be a combinatorial circle, let $P$ be a collection of combinatorial circles with a covering map $\s:P \imto S$, and let $\Ga$ be a graph. Let $\lambda:P\to \Ga$ be a combinatorial map of graph (that is, a morphism of graphs - vertices map to vertices and edges map to edges). Suppose that no edge of $\Ga$ is covered by exactly one edge of $P$ (so every edge is either covered at least twice or not at all). Let $\Ga_W$ be the pushout $\sqcup_P S$, so we have a commutative diagram:
$$
\xymatrix{
P \ar[r]^\s \ar[d]^{\lambda}  &S \ar[d]^w\\
\Ga\ar[r]^{h} &\Ga_W
}
$$
For any graph $G$, denote its set of edges by $E_G$. Suppose that the natural map $E_P \to E_\Ga\times E_S$ is injective, and that $w_*([S])$ is not a proper power in $\pi_1(W)$. Then
$$
\chi(\Ga) \le \chi(\Ga_W) + 1 - \deg \s.
$$
\end{thm}
\begin{remark}
This is a slightly weaker statement than the original statement of \cite{WiltonLouder}: the requirement that no edge of $\Ga$ is covered exactly once by $P$ may be replaced by a less restrictive one (in the language of \cite{WiltonLouder}, we require that $P\to \Ga$ is dependent instead of weakly dependent).
\end{remark}
In our case, we will take $P$ to be the collection of boundary words of an admissible pair. Since we allow boundary components that map to $1$, it is useful to introduce a variation of the above theorem, in which we allow some components of $P$ to be points, mapping to a fixed basepoint chosen on $S$ - that is, "degree 0" coverings of $S$. Such components do not affect $\Ga$, but do change $\Ga_W$ - the vertices which are the images of the point components of $P$ are glued to the basepoint of $\Ga_W$. We claim that:
\begin{thm}\label{thm:dep2}
Let $S$ be a combinatorial circle with some chosen basepoint $o$, let $P$ be a collection of combinatorial circles and points which has at least once circle component, and let $\s:P\imto S$ be a map which is a covering of $S$ when restricted to the circle components of $P$ and maps the point components of $P$ to $o$. Let $\Ga$ be a graph, and suppose we have a map $\lambda:P\to\Ga$, such that no edge of $\Ga$ is covered by exactly one edge of $P$. Let $\Ga_W$ be the pushout $\Ga\sqcup_P S$. Suppose that the natural map $E_P \to E_\Ga\times E_S$ is injective, and suppose that $w_*([S])$ is not a proper power in $\pi_1(\Ga_W)$. Let $\ell_0$ be the number of point components of $P$. Then
$$
\chi(\Ga) - \ell_0 \le \chi(\Ga_W) + 1 - \deg{\s}.
$$
\end{thm}
\begin{proof}
Let $P'$ be the union of the cycle components of $P$, and consider $\Ga_W' = \Ga\sqcup_{P'} S$. By Theorem \ref{thm:dep}, we have that $\chi(\Ga) \le \chi(\Ga'_W) + 1 - \deg{\s}$ (note that restricting $\s$ to $P'$ does not change its degree). To obtain $P$ from $P'$, we need to add $\ell_0$ point components. The affect this has is only on the pushout $\Ga'_W$ - each point component of $P$ identifies a vertex of $\Ga$ with $o$. In other words, we take (at most) $\ell_0$ points of $\Ga'_W$ and glue them to $o$. Each such point decreases the Euler characteristic by $1$, so $\chi(\Ga_W) \ge \chi(\Ga'_W) - \ell_0$. Hence $\chi(\Ga) - \ell_0 \le (\chi(\Ga'_W) + 1 - \deg{\s}) - \ell_0 \le \chi(\Ga_W) + 1 - \deg{\s}$. 
\end{proof}

To relate the dependence theorem to our case, consider an admissible pair $(\S, f)$ for $w^{m_1}, ..., w^{m_\ell}$. As we are allowed to change $f$ by an homotopy relative to the basepoints $v_i$ on $\partial\Sigma$, we may assume that $f|_{\partial\Sigma}$ maps the $i$-th component to the reduced form of $w^{m_i}$. Now, choose one point $p_x$ on each edge $x$ of $\Om = \bigvee^r S^1$, and consider the set $C = f^{-1}(\{p_{x_1},...,p_{x_r}\})$. Up to a change of $f$ by homotopy (without changing it on the boundary), this is a disjoint union of arcs and curves on $\S$, and furthermore for each $x$ and for each arc or curve $\gamma\in f^{-1}(p_x)$, there is a tubular neighborhood $N(\gamma)$ such that the two components of $N(\gamma)\backslash \gamma$ are mapped to different ``sides'' of $p_x$ (in other words, in the terminology of \cite[Definition 3.1]{PuderMageeUn}, the map $f$ is transverse to every $p_{x_i}$). For any edge $x$ of $\Om$, we call an arc or a curve in $f^{-1}(p_x)$ an $x$-arc or an $x$-curve. We now construct a graph $\Ga$ by taking a vertex for each component of $\S - C$, and an edge for any arc or curve of $C$, connecting the vertices corresponding to the two components adjacent to it (if they are the same component we add a self-loop). In particular, $\Ga$ is equipped with a map $\Ga\to \Om$ (sending all vertices to $o$ and an edge corresponding to an $x$-arc or an $x$-curve to the corresponding edge $x$ in $\Om$). We construct a map $\S \to \Ga$ as follows: first, take (disjoint) tubular neighborhoods for each curve and arc of $C$, and let $N$ be their union. There is a natural bijection between the components of $\S-N$ and $\S-C$. We use it to map each component of $\S-N$ to the corresponding vertex of $\Ga$. Let $\gamma$ be an arc or a curve in $C$. It corresponds to some edge $e$ of $\Gamma$. By definition, its tubular neighborhood is homeomorphic to $\gamma\times [-1, 1]$. Note that $\gamma\times\{-1\}$ is at the boundary of one component of $\S-N$ and $\gamma\times \{1\}$ is at the boundary of another component of $\S - N$ (not necessarily different from the first one). The edge $e$ connects the corresponding vertices. If we parametrize $e$ by $[-1,1]$, we can map the tubular neighborhood $\gamma\times[-1,1]$ to $e$ by composing the projection $\gamma\times[-1,1]\to [-1,1]$ with this parametrization. Repeating this for all curves and arcs of $C$ results in the construction of a surjective continuous map $\S \to \Ga$.
Furthermore, the map $f:\S \to \Om$ is homotopic to the composition of this map $\S \to \Ga$ with the natural map $\Ga\to \Om$ constructed before.

Finally, let $P$ be a collection of circles representing the boundary of $\S$. More precisely, for any boundary component $\partial$ of $\S$, construct a combinatorial circle by taking a vertex for any component of $\partial - C$ (which is a collection of segments), and adding an edge for any point in $\partial \cap C$, connecting the vertices of the adjacent segments (in particular, if $\partial$ is mapped to $1$ under $f$, we get a point component). Note that $P$ is equipped with a natural map $P \imto \Ga$, induced from the composition $\partial \S \to \S \to \Ga$. As any edge of $\Ga$ corresponds either to an arc, which has two endpoint, or to a curve, which has no endpoints, every edge of $\Ga$ is covered by either zero or two edges of $P$. In particular, no edge of $\Ga$ is covered by exactly one edge of $P$. Take $S$ to be a combinatorial circle representing $w$ (so it is equipped with an immersion $S\imto \Om$). We have a natural covering map $P \imto S$, with the $i$-th component of $P$ covering $S$ with degree $|m_i|$. In particular, the total degree of the cover is $\deg\s = \sum |m_i|$.

\begin{prop}\label{prop:woforbidden}
Let $(\S, f)$ be as in Proposition \ref{prop:chiS} and let $\Ga$, $S$, $P$ be as above. Suppose $E_P \to E_\Ga\times E_S$ is injective. Then one of the following holds:
\begin{enumerate}
    \item $\chi(\S) < 1 - \pi(w)$
    \item $\chi(\S) = 1 - \pi(w)$ and $\ell = 1$ and $m_1 = \pm 1$, so $(\S, f)$ is an admissible pair for $w$ (or $w^{-1}$)
    \item $\chi(\S) = 1 - \pi(w)$ and $m_1 = ... = m_\ell = 0$
\end{enumerate}
\end{prop}

In the next subsection we show how Proposition \ref{prop:chiS} follows from this proposition, by analyzing the cases where $E_P \to E_\Ga\times E_S$ is not injective. In the rest of this subsection, we prove this proposition. We begin by proving the following:
\begin{prop}\label{prop:Gw}
Let $P, \Ga, S$ be as in Theorem \ref{thm:dep2}. Suppose $\Ga$ is connected. Let $\Om$ be a graph, and suppose there are maps $g:\Ga\to \Om$, $u:S\to\Om$, agreeing on $P$. In particular, there is a well-defined map $r:\Ga_W \to \Om$. Let $o$ be a basepoint of $S$, and let $U$ be a collection of vertices of $P$ mapping to $o$, with exactly one on each component of $P$. Let $V$ be the image of $U$ in $\Ga$. By abuse of notation, we denote $w(o)$ by $o$ as well, and we also let $u$ denote the class $u_*[S] \in \pi_1(\Om)$ (after choosing an arbitrary orientation of $S$). Then
$$
r_*\pi_1(\Ga_W, o) = \grp{g_*\Pi_1(\Ga, V), u}.
$$

$$
\xymatrix{
P \ar[r]^\s \ar[d]^{\lambda}  &S \ar[d]^w\ar@{-->}@/^/[ddr]^u\\
\Ga\ar[r]^{h}\ar@{-->}@/_/[drr]^g &\Ga_W\ar@{-->}[dr]^{r} \\
& & \Om
}
$$
\end{prop}

\begin{proof}[Proof of Proposition \ref{prop:Gw}]
Denote by $O_\Ga$ the set of vertices of $\Ga$ that are mapped to $o\in\Ga_W$ by $h$. Observe that the map $g$ descends to a well defined map $\Ga/O_\Ga\to\Om$. Our strategy will be to show the chain of equalities
$$
r_*\pi_1(\Ga_W, o) = g_*\pi_1(\Ga/O_\Ga) = \grp{g_*\Pi_1(\Ga, V), u}.
$$

We begin by showing that $r_*\pi_1(\Ga_W, o) = g_*\pi_1(\Ga/O_\Ga)$. To this end, consider the space obtained from $(P\times[0,1])\sqcup \Ga\sqcup S$ by gluing $P\times\{0\}$ onto $\Ga$ according to $\lambda$ and gluing $P\times\{1\}$ onto $S$ according $\sigma$ (see Figure \ref{subfig:GwX}). The resulting space $X$ is homotopy equivalent to $\Ga_W$, since $\Ga_W$ can be obtained from $X$ by collapsing $P\times[0,1]$. Thus $\pi_1(\Ga_W, o) \cong \pi_1(X, o_X)$, where $o_X$ is the point of $X$ which is the image of $o\in S$. Moreover, there is a well-defined map $q:X\to \Om$, and $q_*\pi_1(X, o_X) = r_*\pi_1(\Ga_W, o)$. Hence it is enough to show that $q_*\pi_1(X, o_X) = g_*\pi_1(\Ga/O_\Ga)$. We proceed by studying the structure of $X$. Let $O_P$ be the set of vertices of $P$ which map to $o\in S$. They divide the components of $P$ into segments $\{I_j\}_{j\in\mathcal{J}}$, where for each $I_j$, the map $\sigma|_{I_j}:I_j \to S$ covers $S$ exactly once. Moreover, the image of $O_P$ in $\Ga$ is exactly $O_\Ga$. Thus the image of $O_P\times [0,1]\subseteq P\times [0,1]$ in $X$ consists of segments connecting the points of $O_\Ga$ to $o_X$. As for the segments $I_j\subseteq P$, we can think of the image of each $I_j\times [0,1]$ in $X$ as a rectangle, with two of its sides being the segments $\partial I_j\times [0, 1]$, one other side being the image of $I_j$ in $\Ga$, and the last side being $S$ (so two of its vertices are the images of $\partial I_j$ in $\Ga$, and the two other vertices coincide and are in fact $o_X$). Hence we can construct $X$ iteratively as follows. Let $Y$ be the space obtained by taking $\Ga$ and connecting all the points of $O_\Ga$ to a single point $o_X$ (see Figure \ref{subfig:GwY}). Then $X$ is obtained from $Y$ by first adding a circle $S$, and then gluing a rectangle for each segment $I_j$ of $P$, as described above (see Figure \ref{subfig:GwYr}). Observe that almost by definition, $\pi_1(Y, o_X) \cong \pi_1(\Ga/O_\Ga)$. After adding the circle $S$, the fundamental group becomes $\pi_1(\Ga/O_\Ga) * \grp{[S]}$. The image of this in $\pi_1(\Om)$ is simply $g_*\pi_1(\Ga/O_\Ga)$: the image of $[S]$ is $u$, which is already contained in $g_*\pi_1(\Ga/O_\Ga)$ since the image of any segment $I_j$ of $P$ in $\Om$ is $u$ as well, and the image of any such segment in $\Ga/O_\Ga$ is a closed loop. It remains to study the effect of adding the rectangles. It is easy to see that the boundary of each rectangle is mapped to $1$ in $\Om$, so adding them does not affect the image of the $\pi_1(X)$ in $\Om$. All in all, we see that $r_*\pi_1(\Ga_W, o) = q_*\pi_1(X, o_X) = g_*\pi_1(\Ga/O_\Ga)$.

Finally, note we trivially have $g_*\pi_1(\Ga/ O_\Ga) = \grp{g_*\Pi_1(\Ga, O_\Ga)}$. We claim that $\grp{g_*\Pi_1(\Ga, O_\Ga)} = \grp{g_*\Pi_1(\Ga, V), u}$. Indeed, any path between two points $o_1, o_2 \in O_\Ga$ can be decomposed as a path from $o_1$ a point $v_1 \in V$ along the image of some component of $P$, a path from $v_1$ to some $v_2 \in V$, and a path from $v_2$ to $o_2$ along the the image of another component of $P$. Clearly the first and last parts map to powers of $u$, while the middle part is an element of $\Pi_1(\Ga, V)$. Hence $\grp{g_*\Pi_1(\Ga, O_\Ga)} \subseteq \grp{g_*\Pi_1(\Ga, V), u}$, and the opposite inclusion holds since $u\in g_*\pi_1(\Ga/O_\Ga)$, as explained in the previous paragraph.

\begin{figure}
    \centering
    \begin{subfigure}[b]{0.3\textwidth}
    \centering
    \def\svgwidth{1.5\columnwidth}
\begingroup%
  \makeatletter%
  \providecommand\color[2][]{%
    \errmessage{(Inkscape) Color is used for the text in Inkscape, but the package 'color.sty' is not loaded}%
    \renewcommand\color[2][]{}%
  }%
  \providecommand\transparent[1]{%
    \errmessage{(Inkscape) Transparency is used (non-zero) for the text in Inkscape, but the package 'transparent.sty' is not loaded}%
    \renewcommand\transparent[1]{}%
  }%
  \providecommand\rotatebox[2]{#2}%
  \newcommand*\fsize{\dimexpr\f@size pt\relax}%
  \newcommand*\lineheight[1]{\fontsize{\fsize}{#1\fsize}\selectfont}%
  \ifx\svgwidth\undefined%
    \setlength{\unitlength}{172.57648247bp}%
    \ifx\svgscale\undefined%
      \relax%
    \else%
      \setlength{\unitlength}{\unitlength * \real{\svgscale}}%
    \fi%
  \else%
    \setlength{\unitlength}{\svgwidth}%
  \fi%
  \global\let\svgwidth\undefined%
  \global\let\svgscale\undefined%
  \makeatother%
  \begin{picture}(1,0.40733643)%
    \lineheight{1}%
    \setlength\tabcolsep{0pt}%
    \put(0,0){\includegraphics[width=\unitlength,page=1]{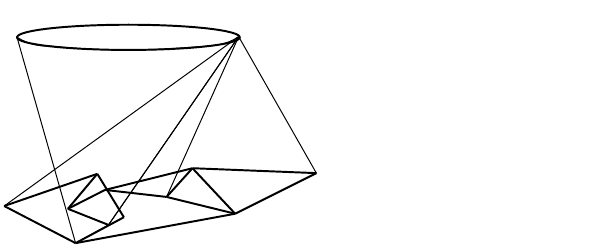}}%
    \put(0.07458855,0.37166386){\makebox(0,0)[lt]{\lineheight{1.25}\smash{\begin{tabular}[t]{l}$S$\end{tabular}}}}%
    \put(0.40739744,0.36415191){\makebox(0,0)[lt]{\lineheight{1.25}\smash{\begin{tabular}[t]{l}$o_X$\end{tabular}}}}%
    \put(0.53364124,0.09463585){\makebox(0,0)[lt]{\lineheight{1.25}\smash{\begin{tabular}[t]{l}$\Gamma$\end{tabular}}}}%
    \put(0.49922602,0.19136617){\makebox(0,0)[lt]{\lineheight{1.25}\smash{\begin{tabular}[t]{l}$P\times\{\frac12\}$\end{tabular}}}}%
    \put(0,0){\includegraphics[width=\unitlength,page=2]{PSG.pdf}}%
  \end{picture}%
\endgroup%

    \caption{The space $X$}
    \label{subfig:GwX}
    \end{subfigure}
    \hfill
    \begin{subfigure}[b]{0.3\textwidth}
    \centering
    \def\svgwidth{\columnwidth}
\begingroup%
  \makeatletter%
  \providecommand\color[2][]{%
    \errmessage{(Inkscape) Color is used for the text in Inkscape, but the package 'color.sty' is not loaded}%
    \renewcommand\color[2][]{}%
  }%
  \providecommand\transparent[1]{%
    \errmessage{(Inkscape) Transparency is used (non-zero) for the text in Inkscape, but the package 'transparent.sty' is not loaded}%
    \renewcommand\transparent[1]{}%
  }%
  \providecommand\rotatebox[2]{#2}%
  \newcommand*\fsize{\dimexpr\f@size pt\relax}%
  \newcommand*\lineheight[1]{\fontsize{\fsize}{#1\fsize}\selectfont}%
  \ifx\svgwidth\undefined%
    \setlength{\unitlength}{136.71889437bp}%
    \ifx\svgscale\undefined%
      \relax%
    \else%
      \setlength{\unitlength}{\unitlength * \real{\svgscale}}%
    \fi%
  \else%
    \setlength{\unitlength}{\svgwidth}%
  \fi%
  \global\let\svgwidth\undefined%
  \global\let\svgscale\undefined%
  \makeatother%
  \begin{picture}(1,0.50468738)%
    \lineheight{1}%
    \setlength\tabcolsep{0pt}%
    \put(0,0){\includegraphics[width=\unitlength,page=1]{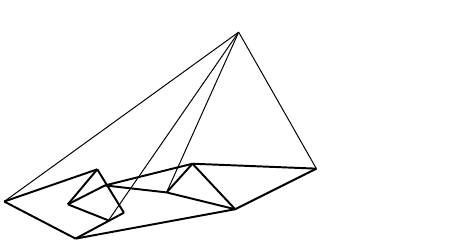}}%
    \put(0.51424653,0.45965889){\makebox(0,0)[lt]{\lineheight{1.25}\smash{\begin{tabular}[t]{l}$o_X$\end{tabular}}}}%
    \put(0.67360059,0.11945621){\makebox(0,0)[lt]{\lineheight{1.25}\smash{\begin{tabular}[t]{l}$\Gamma$\end{tabular}}}}%
    \put(0,0){\includegraphics[width=\unitlength,page=2]{PSG-y.pdf}}%
  \end{picture}%
\endgroup%

    \caption{The space $Y$}
    \label{subfig:GwY}
    \end{subfigure}
    \hfill
    \begin{subfigure}[b]{0.3\textwidth}
    \centering
    \def\svgwidth{1.5\columnwidth}
\begingroup%
  \makeatletter%
  \providecommand\color[2][]{%
    \errmessage{(Inkscape) Color is used for the text in Inkscape, but the package 'color.sty' is not loaded}%
    \renewcommand\color[2][]{}%
  }%
  \providecommand\transparent[1]{%
    \errmessage{(Inkscape) Transparency is used (non-zero) for the text in Inkscape, but the package 'transparent.sty' is not loaded}%
    \renewcommand\transparent[1]{}%
  }%
  \providecommand\rotatebox[2]{#2}%
  \newcommand*\fsize{\dimexpr\f@size pt\relax}%
  \newcommand*\lineheight[1]{\fontsize{\fsize}{#1\fsize}\selectfont}%
  \ifx\svgwidth\undefined%
    \setlength{\unitlength}{172.57648247bp}%
    \ifx\svgscale\undefined%
      \relax%
    \else%
      \setlength{\unitlength}{\unitlength * \real{\svgscale}}%
    \fi%
  \else%
    \setlength{\unitlength}{\svgwidth}%
  \fi%
  \global\let\svgwidth\undefined%
  \global\let\svgscale\undefined%
  \makeatother%
  \begin{picture}(1,0.40733643)%
    \lineheight{1}%
    \setlength\tabcolsep{0pt}%
    \put(0,0){\includegraphics[width=\unitlength,page=1]{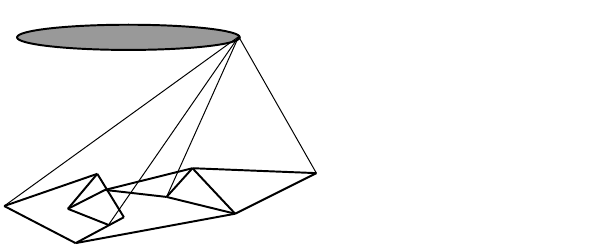}}%
    \put(0.07458855,0.37166386){\makebox(0,0)[lt]{\lineheight{1.25}\smash{\begin{tabular}[t]{l}$S$\end{tabular}}}}%
    \put(0.40739744,0.36415191){\makebox(0,0)[lt]{\lineheight{1.25}\smash{\begin{tabular}[t]{l}$o_X$\end{tabular}}}}%
    \put(0.53364124,0.09463585){\makebox(0,0)[lt]{\lineheight{1.25}\smash{\begin{tabular}[t]{l}$\Gamma$\end{tabular}}}}%
    \put(0.49922602,0.19136617){\makebox(0,0)[lt]{\lineheight{1.25}\smash{\begin{tabular}[t]{l}$P\times\{\frac12\}$\end{tabular}}}}%
    \put(0,0){\includegraphics[width=\unitlength,page=2]{PSG-yx.pdf}}%
  \end{picture}%
\endgroup%

    \caption{Adding one rectangle to $Y$}
    \label{subfig:GwYr}
    \end{subfigure}
    \caption{Some spaces constructed in the proof of Proposition \ref{prop:Gw}. The marked points consist of $o_X$, $O_\Ga$, and $O_P\times \{\frac12\}$}
\end{figure}

\end{proof}

\begin{proof}[Proof of Proposition \ref{prop:woforbidden}]
We begin by analyzing the case where $m_1=\dots=m_\ell=0$. Let $\S'$ be the surface obtained by gluing disks to $\ell - 1$ of the boundaries of $\S$. In particular, $\chi(\S) = \chi(\S') - (\ell - 1)$. We keep the marked points $V$ of $\S$ in $\S'$, but note that they are no longer on the necessarily on the boundary. As the boundaries of $\S$ map to $1$ under $f_*$, we have that $\grp{f_*\Pi_1(\S, V)} = \grp{f_*\Pi_1(\S', V)}$. Denote this group by $H$. Consider now the quotient $\S'/V$. Since $\S'$ is a surface with boundary, it is homotopy equivalent to a graph and hence its $\pi_1$ is free. As $\S'/V$ is obtained by gluing together $\ell$ points of $\S'$, we obtain that $\pi_1(\S'/V)$ is also free, and $\rk \pi_1(\S'/V) = \rk \pi_1(\S') + (\ell - 1)$. On the other hand, $f$ descends to a map $\S'/V\to \Om$, and $f_*\pi_1(\S'/V)$ can be identified with $H$. By assumption $\grp{H,w}$ is a non-cyclic algebraic extension of $w$, and thus by Claim \ref{clm:Hw}, $\rk H \ge \pi(w)$. Thus
$$
\rk \pi_1(\S') + (\ell - 1) = \rk \pi_1(\S'/V) \ge \rk f_*\pi_1(\S'/V) = \rk H \ge \pi(w).
$$
It immediately follows that
$$
\chi(\S) = \chi(\S') - (\ell - 1) = 1 - \rk \pi_1(\S') - (\ell - 1) \le 1 - \pi(w),
$$
as we wanted.

Assume now that not all the $m_i$ are $0$. Let $\partial_0$ be the set of boundary components of $\S$ corresponding to the $m_i$ which are $0$ (so they are mapped to $1$ by $f$) and let $\ell_0$ be the size of $\partial_0$. We begin by showing that 
$$\chi(\S) \le \chi(\Ga) - \ell_0.
$$
To show this, we may assume that $\ell_0= 0$: indeed, let $C$ be the collection of arcs and curves used to define $\Ga$, as above. Note that no arc of $C$ touches $\partial_0$. Hence if we glue disks onto those boundaries ("filling the holes"), $\Ga$ does not change but $\chi(\S)$ increases by $\ell_0$. Hence if we can prove that $\chi(\S) \le \chi(\Ga)$ for this new $\S$, we will be done. Hence we may assume $\ell_0 = 0$. Now, note that we may further assume no curve of $C$ bounds a disk (otherwise we can change $f$ by an homotopy in a way that eliminates this curve from the pre-images $f^{-1}(p_x)$). Replace each curve and arc of $C$ by a small tubular neighborhood, to get a collection of rectangles and annuli. Removing them from $\S$ leaves a collection of connected components $\S_i$. In other words, $\S$ is obtained by taking a collection of surfaces with boundary $\S_i$, and connecting some of their boundaries to each other by a collection of rectangles and annuli. $\Ga$ is obtained by collapsing each $\S_i$ to a point and each rectangle or annulus to an edge. We now do this collapse in an iterative process, and keep track of $\chi$, as follows. First, collapse any rectangle to an arc. This is a homotopy equivalence and thus does not affect $\chi$. Next, replace every annulus by two disks filling its boundary circles (so we "close some holes" in some of the $\S_i$), connected by an arc. This increases $\chi$ by $1$ per annulus, as we add two disks and an arc and remove an annulus (which has $\chi=0$). We are left with a new space $Y$, which is collection of surfaces $\S_i'$, not necessarily with boundary, connected together by some edges. Then
\begin{eqnarray*}
\chi(\S) = \chi(Y) - \#\{\text{annuli we had}\} = \sum_i \chi(\S_i') - \#\{\text{edges of Y}\} - \#\{\text{annuli we had}\} = \\
\sum_i \left(\chi(\S_i') - \frac{1}{2}\#\{\text{annuli that touched } \S_i\}\right) -  \#\{\text{edges of Y}\}.
\end{eqnarray*}
(Note that if an annulus touched the same $\S_i$ on both of its sides it should be counted twice in $\#\{\text{annuli that touched } \S_i\}$).
On the other hand, $\Ga$ is obtained by collapsing each $\S_i'$ to a point, so $\chi(\Ga) = \sum_i 1 - \#\{\text{edges of Y}\}$. Hence to show that $\chi(\S) \le \chi(\Ga)$, it is enough to show that $\chi(\S_i') - \frac{1}{2}\#\{\text{annuli that touched } \S_i\} \le 1$ for any $\S_i'$. This is clear: first, note that the only connected surface (possibly with boundary) that has $\chi > 1$ is the sphere, which has $\chi = 2$. Thus we only have to deal with the $\S_i'$ which are spheres. Since the original $\S_i$ all had boundary, $\S_i'$ can have no boundary only if an annulus was attached to every boundary component of $\S_i$. In this case, $\S_i$ must have had at least two boundary components, otherwise we would have a curve bounding a disk (or a disk component of the original $\S$ which is impossible since $\S$ is connected). Thus the contribution is at most $2 - \frac 12 \cdot 2 = 1$, and we are done. Hence we showed that $\chi(\S) \le \chi(\Ga)$ if $\ell_0 = 0$, or more generally, $\chi(\S) \le \chi(\Ga) - \ell_0$.

Let $\Ga_W$ be as in the dependence theorem. Suppose we could show that $\rk \pi_1(\Ga_W) \ge \pi(w)$ and that the class of $w_*[S]$ is not a proper power in $\pi_1(\Ga_W)$. The first condition guarantees that $\chi(\Ga_W) = 1 - \rk \pi_1(\Ga_W) \le 1 - \pi(w)$. The second condition allows us to apply the dependence theorem (or our generalized version, Theorem \ref{thm:dep2}), and we get:
$$
\chi(\S) \le \chi(\Ga) - \ell_0 \le 1 - \pi(w) + 1 - \sum |m_i|.
$$
If $\sum |m_i| > 1$, $\chi(\S) < 1 - \pi(w)$ and we are done. Otherwise, $\sum |m_i| = 1$ and we must have $\ell = 1$ and $m_1 = \pm 1$. In this case, by the second remark following the statement of Proposition \ref{prop:chiS}, either $\chi(\S) < 1 - \pi(w)$ or $\pi(w)=2\cl(w)$ and $\chi(\S) = 1 - 2\cl(w)$, as desired. Thus the conclusion of the proposition is satisfied in either case.

It remains to show that $\rk \pi_1(\Ga_W) \ge \pi(w)$ and that $w_*[S]$ is not a proper power in $\pi_1(\Ga_W)$. Both claims follow from Proposition \ref{prop:Gw}, as we now explain. Choose a collection of marked vertices $U$ on $P$ as follows: recall that each vertex of $P$ corresponds to a segment of $\partial \S$, and take $U$ to be the set of vertices which correspond to the segments that contain a point of $V$ (clearly there is exactly one on each component of $P$). By abuse of notation, we denote by $V$ their image in $\Ga$ (it is simply the collection of components of $\S - C$ containing a point of $V$). We claim that $\grp{f_*\Pi_1(\Ga, V)} = \grp{f_*\Pi_1(\S, V)}$. Indeed, any path on $\S$ between the points of $V$ gives a path on $\Ga$ between the corresponding points and having the same image, and conversely, any path in $\Ga$ between points of $V$ can be lifted to a path in $\S$ having the same image.

Applying Proposition \ref{prop:Gw} (and noting that $\Ga$ is connected since $\S$ is connected), we have that
$$
h_*\pi_1(\Ga_W, o) = \grp{f_*\Pi_1(\Ga, V), w} = \grp{f_*\Pi_1(\S, V), w}.
$$
Hence $\rk \pi_1(\Ga_W) \ge \rk h_*\pi_1(\Ga_W) = \rk \grp{f_*\Pi_1(\S, V), w}$. Since $\grp{f_*\Pi_1(\S, V), w}$ is an algebraic extension of $w$ by assumption, its rank is at least $\pi(w)$. Thus $\rk \pi_1(\Ga_W) \ge \pi(w)$.

Finally, to see that $w_*[S]$ is not a proper power in $\pi_1(\Ga_W)$, it is enough to show that $h_*w_*[S]$ is not a proper power in $h_*\pi_1(\Ga_W)$. But $h_*w_*[S] = w$ by definition, which is a non-power by assumption, so we are done.
\end{proof}

\subsection{Forbidden matchings}\label{subsec:forbidden}
In this section, we study the injectivity of the map $E_P \to E_\Ga\times E_S$. We show that if it is not injective, we can construct a new surface $\S'$ with $\chi(\S) \le \chi(\S')$ and the new map $E_P \to E_\Ga\times E_S$ being injective (so we can apply Proposition \ref{prop:woforbidden} for $\S'$ and get a bound on $\chi(\S)$).

We begin by analyzing how the map $E_P \to E_\Ga\times E_S$ can fail to be injective. Let $e_1$ and $e_2$ be two edges of $P$ mapping to the same edge in both $\Ga$ and $S$. Recall that edges of $\Ga$ correspond to the arcs and curves of $C$, while the edges of $P$ correspond to endpoints of the arcs of $C$. Furthermore, if an edge of $P$ is the endpoint of some arc of $C$, its image in $\Ga$ is the edge corresponding to this arc. Hence $e_1$ and $e_2$ can map to the same edge $e$ of $\Ga$ if and only if $e$ corresponds to an arc of $C$ and $e_1$, $e_2$ correspond to its endpoints. To understand the images of $e_1$ and $e_2$ in $S$, note first that any edge of $P$ is labelled by some edge of $\Om$ (and hence by a letter of $F_r = \pi_1(\Om)$) - indeed, by definition, each arc of $C$ is mapped to one of our chosen points $p_x$. In fact, each arc comes with an orientation (pulled back from $\Om$), so we can write either $x$ or $x^{-1}$ accordingly. By our construction, if we go along some boundary component of $\S$ and write down the letters corresponding to each arc-endpoint we encounter, we will get the word $w^p$ for some $p$ (or one of its cyclic shifts). In other words, each edge of $P$ is labelled by some $x$ or $x^{-1}$, and each circle of $P$ reads off as some $w^p$. The map to $S$ is obtained by dividing each circle of $P$ to segments representing $w$ and mapping each segment to $S$. Hence, two edges of $P$ map to the same edge of $S$ if and only if they have the same "location relative to $w$". Overall, $e_1$ and $e_2$ map to the same edges of $\Ga$ and $S$ if and only if they are the two endpoints of an arc, which correspond to the same location relative to $w$ in their respective boundary components. In fact, we obtain the following:
\begin{prop}
Let $\S, f, C, P, \Ga, S$ be as before. The arcs of $C$ induce a collection of matchings on the edges of $P$, or equivalently, a collection of matchings for the words $w^{m_1},...,w^{m_\ell}$. The map $E_P \to E_\Ga\times E_S$ is injective if and only if in the induced collection of matchings no letter is matched to another letter which corresponds to the same letter in $w$.
\end{prop}

\begin{example}
Take $w=xy^2x$, and suppose the boundary consists of $xy^2x$ and $x^{-1}y^{-2}x^{-1}$. Then if our collection of arcs matches the first $x$ of the first boundary to the last $x^{-1}$ of the second boundary, the map $E_P \to E_\Ga\times E_S$ will not be injective, since they both correspond to the first $x$ of $w$.
\end{example}

We call a matching that fails to satisfy this property a \emph{forbidden} matching. That is, a matching is forbidden if it matches two corresponding letters of $w$. 

\begin{remark}
Our terminology follows \cite{Magee2021}, where the same property was studied in a similar but slightly different context. Indeed, it is interesting to note that while both this paper and \cite{Magee2021} arrive to the notion of forbidden matching from the study of word measures on $U(n)$, our work focuses on word measures defined by words in free groups, while the work of Magee focuses on word measures defined by words in surface groups.
\end{remark}

We now explain how to simplify pairs with forbidden matchings:

\begin{prop}\label{prop:simplify}
Let $(\S, f)$ be an admissible pair as in Proposition \ref{prop:chiS}, with a chosen underlying system of arcs and curves $C$. If the collection of matchings induced by $C$ is forbidden, then one can construct some $(\S', f')$ such that: 
\begin{enumerate}
    \item $(\S', f')$ satisfies the conditions of Proposition \ref{prop:chiS}.
    \item $\S'$ has strictly less boundary components than $\S$.
    \item $\chi(\S) < \chi(\S')$.
\end{enumerate}
\end{prop}
Proposition \ref{prop:chiS} follows from this proposition (up to Proposition \ref{prop:Gw}, which will be proved in the next section): 
\begin{proof}[Proof of Proposition \ref{prop:chiS}, assuming Proposition \ref{prop:simplify}]
If $E_P \to E_\Ga\times E_S$ is injective, we are done by Proposition \ref{prop:woforbidden}. Otherwise, we can apply Proposition \ref{prop:simplify} and replace $\S$ by some suitable $\S'$. Continue in the same fashion as long as the collection of matchings we have is forbidden. Since the number of boundary components strictly decreases, the process eventually terminates and we are left with some new $(\S_0, f_0)$, with an arc and curve system $C_0$ which induces a non-forbidden collection of matchings. Hence the map $E_P\to E_\Ga\times E_S$ (where $P$, $\Ga$, and $S$ are the ones constructed from $\S_0$) is injective. Furthermore, by our construction, $\chi(\S) < \chi(\S_0)$. By Proposition \ref{prop:woforbidden}, we must have $\chi(\S_0) \le 1 - \pi(w)$. Thus $\chi(\S) < 1 - \pi(w)$ and we are done.
\end{proof}

\begin{proof}[Proof of Proposition \ref{prop:simplify}]
Fix an admissible pair $(\S, f)$ satisfying the assumptions of Proposition \ref{prop:chiS}, with a chosen underlying systems of arcs and curves $C$ which induces a forbidden matchings. In particular, $(\S, f)$ is an admissible pair for $w^{m_1},...,w^{m_\ell}$ (for some integers $m_i$), $\S$ is connected, and $\grp{f_*\Pi_1(\S, V), w}$ is a non-cyclic algebraic extension of $\grp{w}$.

By assumption, the underlying system of arcs and curves $C$ defines a forbidden matching of $w^{m_1},\dots, w^{m_\ell}$. The letters of the words $w^{m_i}$ can be identified with edges of $P$, and so correspond to endpoints of arcs of $C$, as explained in the beginning of this section. Hence we have some arc $a$ whose two endpoints have the same location relative to $w$ in their respective boundaries. Denote those boundaries by $\partial_1$, $\partial_2$ (they must be different, since corresponding edges on the same boundary cannot be matched as they have the same orientation). 

Consider the surface $\S'$ obtained by cutting $\S$ along $a$. This reduces the number of boundary components of $\S$ by $1$, since $\partial_1$ and $\partial_2$ are now joined together and form a new boundary component $\partial_{12}$. Since the endpoints of $a$ had the same image relative to $w$, the image $f_*[\partial_{12}]$ is still a power of $w$, hence $\S'$ is an admissible surface for some collection of powers of $w$. Clearly $\S'$ is connected, since $\S$ was connected. Furthermore, we have $\chi(\S') = \chi(\S) + 1 > \chi(\S)$. It remains to check that $\grp{f_*\Pi_1(\S', V'), w}$ is a non-cyclic algebraic extension of $w$. For this, we first need to explain how $V'$ is defined. Note that when joining $\partial_1$ and $\partial_2$ together to form $\partial_{12}$, we naturally obtain two basepoints on $\partial_{12}$. Hence in $V'$ we will just drop one of them. For any path $\gamma_{12}$ between $v_1$ and $v_2$, we have that $\grp{f_*\Pi_1(\S', V)} = \grp{f_*\Pi_1(\S', V'), f_*[\gamma_{12}]}$. Indeed, it does matter any path we take since the difference between any two paths will be a loop around our existing basepoint in $V'$. In particular we can choose the path that goes along $\partial_{12}$. The image of this path is a power of $w$ (again, since the arc $a$ defined a forbidden matching), so $\grp{f_*\Pi_1(\S', V), w} = \grp{f_*\Pi_1(\S', V'), w}$. However, there is also a difference between $\Pi_1(\S, V)$ and $\Pi_1(\S', V)$ - any path in $\S$ that crossed the arc $a$ cannot be defined in $\S'$. Let $\gamma$ be such a path. We may decompose it to a path $\gamma'$ that goes until $v_1$ (or $v_2$) without crossing $a$, and a path $\gamma''$ that goes from $v_1$ and crosses $a$ immediately. Consider the path $\partial_1 \gamma''$ - it is homotopic to the path that goes around $\partial_1$ from the other side, instead of crossing $a$. Hence this path also exists on $\S'$ (if it crosses $a$ again in a later point we can repeat this process). In other words, we see that we can replace paths in $\Pi_1(\S, V)$ that cross $a$ by paths in $\Pi_1(\S', V) * \grp{\partial_1}$. All in all, we get that
$$
\grp{f_*\Pi_1(\S, V)} = \grp{f_*\Pi_1(\S', V'), f_*[v_1 v_2], f_*[\partial_1]}
$$
so $\grp{f_*\Pi_1(\S', V'), w}=\grp{f_*\Pi_1(\S, V), w}$. In particular, $\grp{f_*\Pi_1(\S', V'), w}$ is a non-cyclic algebraic extension of $w$.
\end{proof}

Thus Proposition \ref{prop:simplify} is proved and we can eliminate forbidden matchings. Hence, as explained, Proposition \ref{prop:chiS} is proved.

\subsection{The coefficient of $n^{1-\pi(w)}$}\label{subsec:coeff}

So far, we proved Proposition \ref{prop:chiS}. As a consequence, we showed that for $T = \xi_{m_1}\cdots \xi_{m_\ell}$ (where $m_i\neq 0$ for any $i$),
$$
\EE_w[T] = \grp{T, 1} + C(T, w)\cdot n^{1-\pi(w)} + O(n^{-\pi(w)})
$$
for some $C(T, w)$. To prove Theorem \ref{thm:EwTpi}, we have to show that $C(T, w) = \grp{T, \tr + \overline{\tr}} \cdot |\CommCrit(w)|$. We do so by showing the following two propositions:
\begin{prop}\label{prop:CTw}
If $\pi(w) = 2\cl(w)$ then $C(T, w) = \grp{T, \tr+\overline{\tr}} \cdot C(w)$, where $C(w)$ is the coefficient of $n^{1-2\cl(w)}$ in $\EE_w[\tr]$. 
\end{prop}
\begin{prop}\label{prop:Cw}
If $\pi(w) = 2\cl(w)$ then the coefficient of $n^{1-2\cl(w)}$ in $\EE_w[\tr]$ is $|\CommCrit(w)|$.
\end{prop}

\begin{proof}[Proof of Proposition \ref{prop:CTw}]
According to our discussion following Proposition \ref{prop:chiS}, $C(T, w)$ is the total contribution of admissible pairs $(\S, f)$ where $\S$ has one component $\S_0$ such that $(\S_0, f|_{\S_0})$ is admissible for $w^{\pm 1}$ and $\S_0$ is of genus $\cl(w)$, and all other components have $\chi = 0$. As the only surface with boundary with $\chi = 0$ is the annulus, all other components of $\S$ must be annuli. In particular, each such component has exactly two boundary components, which must map to $w^p$, $w^{-p}$ for some $p$ (the boundaries of an annulus are homotopic, so they should have the same image, but we read them with opposite orientations).

Let $(\S, f)$ be an admissible pair contributing to $C(T, w)$. Then its contribution is  $\chi^{(2)}(\MCG(f))$. Write $\S = \S_0 \sqcup \bigsqcup_i \S_i$, where $\S_0$ is the surface with boundary $w^{\pm 1}$ and $\S_i$ are the annuli. We have that
$$
\MCG_\S(f) = \MCG_{\S_0}(f|_{\S_0}) \times \prod_i \MCG_{\S_i}(f|_{\S_i}).
$$
However, the mapping class group of an annulus $\S_i$ is isomorphic to $\ZZ$, and it is generated by Dehn twists. It is easy to see that Dehn twists cannot fix $f|_{\S_i}$, which is map factoring through the projection $S^1\times I \to S^1$ (for example the arc systems of $f$ and $\phi(f)$ for any nontrivial $\phi\in\MCG(\S)$ are not isotopic). Hence $\MCG_{\S_i}(f|_{\S_i}) = 1$, and so $\MCG_\S(f) = \MCG_{\S_0}(f|_{\S_0})$.

Write $C(T, w) = C_+(T, w)+C_-(T, w)$, where $C_+(T,w)$ consists of the contributions of those pairs where the non-annulus component has $w$ as its boundary, while $C_-(T, w)$ corresponds to a $w^{-1}$ boundary. We will show that $C_\pm(T, w) = \grp{T, \xi_{\pm 1}} \cdot C(w)$. We focus only on $C_+(T, w)$, as the argument for $w^{-1}$ is completely analogous (and in fact follows by replacing $w$ by $w^{-1}$).

Every admissible pair $(\S, f)$ contributing to $C_+(T,w)$ induces a partition of $w^{m_1},\dots, w^{m_\ell}$ into a collection of pairs $(w^p, w^{-p})$ and a single $w$. Fixing such a partition, what is the total contribution of admissible pairs inducing it? For each $(w^p, w^{-p})$, there are exactly $p$ equivalence classes of annuli with those boundaries (corresponding to the offset between the marked points, see also \cite[Corollary 1.13]{PuderMageeUn}). As for the $w$ component, we go over all admissible pairs for $w$ of genus $\cl(w)$. As each contributes $\chi^{(2)}_\S(f) = \chi^{(2)}_{\S_0}(f|_{\S_0})$, their total contribution (fixing a choice of annuli) is exactly the coefficient $C(w)$ of $n^{1-2\cl(w)}$ in $\EE_w[\tr]$. Hence, the total contribution of pairs inducing a given partition is $\sqrt{\prod_i |m_i|}\cdot C(w)$. In particular, it is independent of the partition.

We now count how many possible partitions there are. Let $a_p$ be the number of the $m_i$ which are equal to $p$. For there to be a partition, we must have $a_p = a_{-p}$ for all $p\neq \pm 1$ and $a_1 = a_{-1} + 1$. The total number of partitions in this case is $a_1! a_2! a_3! \cdots$ (note that it is $a_1!$ since we choose which $w$ is singled out). If we set $b_p = a_{-p}$ for $p > 1$ and $b_1 = a_{-1} + 1$, we obtain that there are no possible partitions unless $a_p = b_p$ for all $p > 0$. We can also write $\sqrt{\prod_i |m_i|} = \prod_{p > 0} p^{a_p}$. Combining our results, we obtain that
$$
C_+(T, w) = \left(\prod_{p > 0} \delta_{a_p b_p} a_p!p^{a_p}\right) \cdot C(w).
$$
On the other hand, \cite[Theorem 2]{DiaconisShershahani} states that for large enough $n$,
$$
\int_{U(n)} \xi_1^{a_1}\xi_{-1}^{b_1} \cdots \xi_k^{a_2}\xi_{-2}^{b_k}\cdots = \prod_{p > 0}\delta_{a_p b_p} a_p! p^{a_p}.
$$
As $\grp{T, \tr} = \int T\cdot\overline{\tr}$, we see that $\xi_{-1}=\overline{\tr}$ appears $a_{-1} + 1$ times in the integrand. Hence the equations coincide, and 
$$
C_+(T, w) = \grp{T, \tr} \cdot C(w).
$$
The argument for $C_-(T, w)$ is completely analogous, as mentioned before.
\end{proof}
\begin{proof}[Proof of Proposition \ref{prop:Cw}]
According to Theorem \ref{thm:recap}, 
$$
\EE_w[\tr] = \sum_{(\S, f)} \chi^{(2)}(\MCG(f)) n^{\chi(\S)}
$$
where the sum goes over all equivalence classes of admissible pairs for $w$. The pairs contributing to the coefficient of $n^{1-2\cl(w)}$ are exactly the ones where $\S$ is of genus $\cl(w)$. We claim that such pairs are in bijection with commutator critical subgroups, and that the coefficient of each pair is exactly $1$.

Let $(\S, f)$ be an admissible pair for $w$, where $\S$ is of genus $g=\cl(w)$. Since $\S$ is a surface with one boundary component, its fundamental group is free of rank $2g$, and its generators $a_1,...,a_g,b_1,...,b_g$ can be chosen so that the boundary is homotopic to the loop defined by the surface word $[a_1,b_1]\cdots [a_g, b_g]$. Applying $f$, we see that $f_*(a_1),...,f_*(a_g), f_*(b_1),...,f_*(b_g)$ is a solution to
$$
[u_1, v_1] \cdots [u_{\cl(w)}, v_{\cl(w)}] = w.
$$
Hence, according to the discussion in Subsection \ref{subsec:alg-intro}, $f_*\pi_1(\S) = \grp{f_*(a_1),...,f_*(a_g), f_*(b_1),...,f_*(b_g)}$ is a commutator critical subgroup of $w$. If $(\S', f')$ is an admissible pair equivalent to $(\S, f)$, we must have $f'_*\pi_1(\S') = f_*\pi_1(\S)$. Indeed, there is a homeomorphism $\rho:\S\to \S'$ such that $f=f'\circ \rho $, so $f_*\pi_1(\S) = f'_* \rho*(\pi_1(\S)) = f'_* \pi_1(\S')$. Summarizing, we see that we can associate a commutator critical subgroup for each equivalence class of admissible pairs. Furthermore, note that for each such admissible pair $(\S, f)$, the map $f_*:\pi_1(\S) \to \pi_1(\Om)$ is injective (since it maps the generators to a free basis of the image). Hence according to \cite[Lemma 5.1]{PuderMageeUn}, $\MCG(f)$ is trivial and so $\chi^{(2)}(\MCG(f)) = 1$. To finish our proof, it remains to show that every commutator critical subgroup is obtained from some admissible pair and that non-equivalent admissible pairs give rise to different commutator critical subgroup.

Let $H\in \CommCrit(w)$ of $w$ and let $g=\cl(w)$. We will construct an admissible pair $(\S, f)$ with $\S$ of genus $g$ and $f_*\pi_1(\S) = H$. By definition, $w=[u_1,v_1]\cdots [u_g, v_g]$ for some generating set $u_1,...,u_g,v_1,...,v_g$ of $H$. Let $\S$ be the surface of genus $g$ with one boundary component. The surface $\S$ is homotopy equivalent to the graph $\bigvee^{2g} S^1$, with the boundary being homotopic to the loop $[a_1, b_1] \cdots [a_g, b_g]$ where the $a_i$ and $b_i$ are the circles of the bouquet. Hence if we define a map $f:\S \to \Om$ by the map $\bigvee^{2g} S^1 \to \Om$ sending $a_i$ to $u_i$ and $b_i$ to $v_i$, we will have that $f_*[\partial\S] = w$, and clearly $f_*\pi_1(\S) = H$, so we are done.

Finally, it remains to show that different equivalence classes give rise to different commutator critical subgroups. Let $H\in\CommCrit(w)$ and suppose two admissible pairs $(\S, f)$ and $(\S', f')$ satisfy that $\text{genus}(\S)=\text{genus}(\S')=\cl(w)$ and $f_*\pi_1(\S) = f'_*\pi_1(\S') = H$. Our goal is to show that they are in fact equivalent. First, as $\S$ and $\S'$ have the same genus, they are homeomorphic. If $\rho:\S \to \S'$ is some homeomorphism between them, we can replace $(\S', f')$ by the equivalent pair $(\S, f'\circ \rho)$. Hence without loss of generality we may assume $\S' = \S$. Let $a_1,...,a_g,b_1,...,b_g$ be generators of $\pi_1(\S)$ so that $[\partial\S] = [a_1,b_1]\cdots[a_g, b_g]$. The subgroup $H=f_*\pi_1(\S) = f'_*\pi_1(\S)$, which is also isomorphic to $\pi_1(\S)$, is free of rank $2g$, and both $f_*(a_1),...,f_*(a_g), f_*(b_1),...,f_*(b_g)$ and $f'_*(a_1),...,f'_*(a_g), f'_*(b_1),...,f'_*(b_g)$ are free bases of it (according to our discussion in the beginning of the proof). Hence there is some $\phi\in\Aut(H)\cong \Aut(\pi_1(\S))$ mapping the first basis to the second. Furthermore, the surface word $w$ is fixed by $\phi$. On the other hand, according to the Dehn-Nielsen-Baer theorem, $\MCG(\S)$ is isomorphic the subgroup of $\Aut(\pi_1(\S))$ fixing $[\partial \S]$ (see \cite[Theorem 2.4]{MageeePuder2015} for a discussion of the Dehn-Nielsen-Baer theorem for surfaces with boundary, or \cite[Chapter 8]{FarbMarg} for the classical version). Hence if we let $\rho\in \MCG(\S)$ be the homeomorphism corresponding to $\phi$, we will have that $f'_* =\rho_*\circ f_*$ by construction. By Proposition \ref{prop:fstar}, $f'\sim \rho\circ f$ (note that in this case $V$ is just one point so $\Pi_1=\pi_1$). Hence $(\S, f)$ and $(\S', f')$ are equivalent pairs (by the homeomorphism $\rho$).
\end{proof}

\section{A better bound for $\EE_w[\tr(A)\tr(A^{-1})]$}\label{sec:pfthm2}
In this section we show that we can obtain a better bound on the asymptotic behaviour of $\EE_w[\chi]$ when $\chi$ is the character $A\mapsto \tr(A)\tr(A^{-1}) - 1$. More precisely, we show that 
$$
\EE_w[\xi_1\xi_{-1}] = 1 + O(n^{2(1-\pi(w))}).
$$

According to \cite[Theorem 2]{DiaconisShershahani}, $\grp{\xi_1\xi_{-1},1} = 1$. Hence, according to Corollary \ref{cor:EwT1}, it is enough to prove that $\chi(\S) \le 2(1-\pi(w))$ for any admissible pair $(\S ,f)\in \mathcal{S}(w, w^{-1})$, in the notation of the corollary.

If $\S$ is not connected, this is quite simple, and in fact holds for any non-connected admissible pair for $(w, w^{-1})$, not just those in $\mathcal{S}(w,w^{-1})$. Indeed, since $\S$ has two boundary components, it must have exactly two connected components, say $\S_1$ and $\S_2$. Each $\S_i$ is a connected surface with one boundary component, which maps to either $w$ or $w^{-1}$ under $f_*$. Hence each $\S_i$ must have genus at least $\cl(w)$ and so $\chi(\S_i) \le 1 - 2\cl(w) \le 1 - \pi(w)$. Thus $\chi(\S) = \chi(\S_1) + \chi(\S_2) \le 2(1-\pi(w))$, as desired. 

If $\S$ is connected, then the fact that $(\S, f) \in \mathcal{S}(w,w^{-1})$ simply means that $\grp{f_*\Pi_1(\S, V)}$ is a non-cyclic algebraic extension of $w$. Thus we prove:

\begin{prop}\label{prop:piwwinv}
If $(\S, f)$ is an admissible pair for $(w, w^{-1})$, $\S$ is connected, and $\grp{f_*\Pi_1(\S, V)}$ is a non-cyclic algebraic extension of $w$, then
$$
\chi(\S) \le 2(1 - \pi(w)).
$$
\end{prop}
\begin{proof}
The surface $\S$ has two boundary components, mapping to $w$ and $w^{-1}$. Consider the closed surface $\S_0$ obtained by gluing the boundaries of $\S$ together. Note that $\chi(\S) = \chi(\S_0)$ since cycles have $\chi = 0$. The orientation of $\S$ guarantees that $f$ descends to a well-defined map on $\S_0$, mapping the circle that was once the boundary to $w$ (or $w^{-1}$, depending on a choice of orientation for the circle).

Let $V=\{v_0, v_1\}$ be the marked points on the boundaries of $\S$. In $\S_0$ they are glued together, and we call this point $v_0$. We claim that $f_*\pi_1(\S_0, v_0) = \grp{f_*\Pi_1(\S, V)}$. To see why, we will construct $\S_0$ in two steps. First, glue $v_0$ and $v_1$ together to obtain a new space $\S'$. We have that $\grp{f_*\Pi_1(\S, V)}=f_*\pi_1(\S', v_0)$: every path between the points of $V$ map to a loop in $\S'$ based at $v_0$, and every loop at $\S'$ based at $v_0$ can be lifted to a sequence of paths on $\S$ between the points of $V$. Hence after applying $f_*$ they generate the same group. Now, the two boundaries of $\S$ are joined together at a point in $\S'$, and to construct $\S_0$ we have to glue them together. Equivalently, we can glue a rectangle such that two opposite sides of it map to the boundary components and the other two side map to $v_0$. By the van Kampen theorem, it is easy to see the difference between $\pi_1(\S', v_0)$ and $\pi_1(\S, v_0)$ is that the two loops representing the boundary components are now homotopic. As they both map to $w$ under $f$ (with the right choice of orientation), the images $f_*\pi_1(\S, v_0)$ and $f_*\pi_1(\S', v_0)$ are the same.

Thus we have constructed a closed surface $\S_0$ with a map $f:\S_0 \to \Om$ such that $H = f_*\pi_1(\S_0)$ is an algebraic extension of $w$. In other words, we have a surjection from a surface group $\pi_1(\S_0)$ onto a free group $H$, which is an algebraic extension of $w$. However, it is well known that a free quotient of a surface group must have rank at most $g$, where $g$ is the genus of the corresponding surface (see the next proposition, Proposition \ref{prop:surfacequo}, for details). Hence $\rk H \le \text{genus}(\S_0)$. However, $\rk H \ge \pi(w)$, and $\chi(\S) = \chi(\S_0)$, so
$$
\chi(\S) = \chi(\S_0) = 2 - 2\cdot\text{genus}(\S_0) \le 2 - 2\rk H \le 2 - 2\pi(w).
$$
\end{proof}
\begin{remark}
It can be shown that there is only one admissible pair contributing to the term $1=\grp{\xi_1\xi_{-1},1}$. This is the annulus $S^1 \times I$, together with the map to $\Om$ which is the composition $S^1\times I \to S^1 \xrightarrow{w} \Om$ (the last map represents $w$ in $\Om$).
\end{remark}

\begin{prop}\label{prop:surfacequo}
Let $S_g$ be the surface group corresponding to a surface of genus $g$. If $S_g \onto F_r$, then $r \le g$.
\end{prop}
\begin{proof}
This is well known (for example, it appears in \cite[Chapter~I.7]{LyndonSchupp}), but we will sketch a proof anyway for the sake of completeness. Our  description is based on a work of Stallings (\cite{StallingsCorank}).

Let $\S_g$ be a surface of genus $g$, so $S_g \cong \pi_1(\S_g)$. Any surjective map $S_g \onto F_r$ gives rise to a surjection $f:\S_g \onto \Om$. Take the pre-images of the midpoints of the edges of $\Om$. This is a collection of simple curves $C_i$ on $\S_g$ (perhaps after perturbing $f$ a little by a homotopy). They separate $\S_g$ into several connected components, each being mapped to $\Om$ by a map homotopic to the trivial map. However, when crossing each curve we complete a loop around the relevant edge of $\Om$. Thus, we can take a small tubular neighborhood around each curve and assume it maps to the relevant edge, and that all the other points, outside these neighborhoods, map to the basepoint.

Consider the graph $\Ga$ whose vertices correspond to the connected components of $\S_g - \cup C_i$ and whose edges corresponds to the curves $C_i$. The map $\S_g \to \Om$ then factors through $\Ga$ - the connected components map to the vertices, and the tubular neighborhoods map to the edges (the curves themselves will be mapped to, say, the midpoints of the edges). 

Since we have a surjection $\Ga\to \Om$, we must have that $\rk \pi_1(\Ga) \ge \rk \pi_1(\Om) = r$. Hence $\chi(\Ga) \le 1 - r$. On the other hand, we claim that $\chi(\S) \le 2\chi(\Ga)$. This will complete the proof, as $\chi(\S) = 2 - 2g$ and so we will get that $2 - 2g \le 2 (1 - r)$. To bound $\chi(\S)$, let the connected components of $\S_g - \cup C_i$ be denoted $\S_j$, and suppose that $\S_j$ has genus $g_j$ and exactly $b_j$ boundary components. Then 
$$
\chi(\S_g) = \sum_j \chi(\S_j) = \sum_j (2 - 2g_j - b_j) \le \sum_j (2 - b_j) = 2\#\{\S_j\} - \sum_j b_j
$$
Note that by definition, $\#\{S_j\}$ is the number of vertices of $\Ga$ and $\sum_j b_j$ is twice the number of edges of $\Ga$, so
$$
\chi(\S_g) \le 2V_\Ga- 2E_\Ga= 2\chi(\Ga),
$$
and we are done.
\end{proof}
\begin{remark}
We sketch a more conceptual argument for the last part of the proof, namely, after constructing $\Ga$: every cycle of $\Ga$ gives a handle of $\S_g$, and thus $g \ge \rk \pi_1(\Ga)$, but we showed that $\rk \pi_1(\Ga) \ge r$.
\end{remark}

\bibliographystyle{amsalpha}
\bibliography{refs}

\end{document}